\documentclass[11pt]{amsart}

\usepackage{amscd, amssymb}
\usepackage{enumerate}
\usepackage[all]{xypic}
\usepackage{multirow}
\usepackage{hhline}
\usepackage{verbatim}
\usepackage{mathdots}
\usepackage{comment}
\usepackage{tikz-cd}
\usepackage[titletoc]{appendix}
\usepackage{stmaryrd}
\usepackage{graphicx}
\usepackage{mathtools}
\usepackage{leftidx}
\usepackage{mathabx }
\usepackage{amsfonts}
\usepackage{mathrsfs}

\usepackage{dynkin-diagrams}
\usepackage{xcolor}

\usepackage{enumitem}

\usepackage[colorlinks=true, linkcolor=blue,
 citecolor=blue, urlcolor=blue]{hyperref}

\theoremstyle{plain}
\newtheorem{Thm}[equation]{Theorem}
\newtheorem*{Thm*}{Theorem}
\newtheorem{Cor}[equation]{Corollary}
\newtheorem{Prop}[equation]{Proposition}
\newtheorem{Lem}[equation]{Lemma}

\theoremstyle{definition}
\newtheorem*{Rmk}{Remark}
\newtheorem*{Rmks}{Remarks}
\newtheorem{Def}[equation]{Definition}

\numberwithin{equation}{section}

\newcommand{\C}{\mathbb{C}}

\newcommand{\Q}{\mathbb{Q}}

\newcommand{\Z}{\mathbb{Z}}

\newcommand{\pr}{{\operatorname{pr}}}

\newcommand{\Mt}{\widetilde{M}}
\newcommand{\Nt}{\widetilde{N}}

\newcommand{\Gt}{\overline{G}}
\newcommand{\Ht}{\widetilde{H}}

\newcommand{\Tt}{\overline{T}}
\newcommand{\Xt}{X_{Q,n}}

\newcommand{\Yt}{Y_{Q,n}}

\newcommand{\Gbbt}{\overline{{\Gbb}}}
\newcommand{\Tbbt}{\overline{{\Tbb}}}

\newcommand{\Gbb}{{\mathbb{G}}}
\newcommand{\Tbb}{{\mathbb{T}}}
\newcommand{\Bbbb}{{\mathbb{B}}}

\newcommand{\gt}{\tilde{g}}

\newcommand{\xt}{\tilde{x}}
\newcommand{\yt}{\tilde{y}}
\newcommand{\zt}{\tilde{z}}
\newcommand{\wt}{\tilde{w}}
\newcommand{\htt}{\tilde{h}}

\newcommand{\yh}{\hat{y}}

\newcommand{\tbb}{\widebar{t}}

\newcommand{\ybb}{\widebar{y}}
\newcommand{\zbb}{\widebar{z}}

\newcommand{\Ind}{\operatorname{Ind}}

\newcommand{\Irr}{\operatorname{Irr}}

\newcommand{\Span}{\operatorname{span}}
\newcommand{\rank}{\operatorname{rank}}
\newcommand{\End}{\operatorname{End}}
\newcommand{\Hom}{\operatorname{Hom}}

\newcommand{\s}{\mathbf{s}}

\newcommand{\Ocal}{\mathcal{O}}

\newcommand{\Hcal}{\mathcal{H}}

\newcommand{\la}{\langle}
\newcommand{\ra}{\rangle}

\newcommand{\st}{\,:\,}
\newcommand{\qand}{{\quad\text{and}\quad}}

\newcommand{\one}{\mathbf{1}}

\title[Pseudo-spherical representations]{Pseudo-spherical representations and types for covers of split tori}

\author{Edmund Karasiewicz}

\author{Shuichiro Takeda}

\allowdisplaybreaks

\begin{document}

\begin{abstract}

	We introduce a notion of pseudo-spherical representations for any Brylinski-Deligne (BD) cover of a split torus over a $p$-adic field. In particular, we do not assume that the cover is tame. We enumerate the pseudo-spherical representations and compute their dimension. 

    When the torus is the maximal split torus of a connected split reductive group $G$, we enumerate a subset of distinguished pseudo-spherical representations, under mild assumptions on the BD datum. We expect these distinguished representations to provide the torus input for a Borel-Casselman Theorem for BD covers of $G$. We have verified this expectation for certain non-tame double covers in a separate work. 

    While investigating the pseudo-spherical representations, we also prove that the restriction of irreducible genuine representations of covering tori to their maximal compact subgroup is multiplicity-free. From this we see that the irreducible genuine representations of this maximal compact subgroup are types in the sense of Bushnell-Kutzko. We conclude by describing the Hecke algebras attached to these types.
    
\end{abstract}

\maketitle

\setcounter{tocdepth}{1}
\tableofcontents


\section{Introduction}


Earlier work on spherical representations of $n$-fold covers of $p$-adic tori, such as \cite{KP84,S04,W09,Mc12,W16_Tori}, mostly focused on the case of tame covers, meaning $\gcd(n,p)=1$. Here the current theory is well-developed; the most comprehensive treatment is due to Weissman \cite{W09, W16_Tori}. Outside of the tame case, also called wild, much less is known, and what is known is mostly restricted to the case of double covers and is based on somewhat ad hoc arguments. e.g., see \cite{LS10a, LS10b, Wood, Takeda_Wood, Karasiewicz}. In particular, in the literature there is no general notion of spherical representations for covers of tori outside of the tame case. 

In this paper we develop a notion of genuine spherical representation for Brylinski-Deligne (BD) covers of split tori over $p$-adic fields without assuming tameness, which we call pseudo-spherical representations, and develop the theory of these representations. 

We begin with some notation and background before describing our results in more detail.

Let $\Tbb$ be a split torus over a $p$-adic field $F$ with cocharacter lattice $Y$. The group of $F$-points $T=\Tbb(F)$ has a unique maximal compact subgroup $T_{0}$. Then a (smooth) $T$-representation $\chi$ is called spherical (or unramified) if $\chi$ contains nonzero vectors fixed by $T_{0}$.

Now assume that the $n$-th roots of unity in $F$ form a cyclic group $\mu_{n}$ of order $n$. Let $\Tt$ be an $n$-fold BD cover of $T$. So $\Tt$ fits into a central extension
\begin{equation}\label{Eq:CExtTor}
    \mu_{n}\hookrightarrow \Tt\twoheadrightarrow T.
\end{equation}
Let $\Tt_{0}\subset \Tt$ be the unique maximal compact subgroup.

For the $\epsilon$-genuine $\Tt$-representations, meaning $\mu_{n}$ acts through a faithful character $\epsilon:\mu_{n}\rightarrow \C^{\times}$, we would like to have an analogous notion of spherical representations. When $\gcd(n,p)=1$ (i.e. $\Tt$ is tame), the extension \eqref{Eq:CExtTor} splits over $T_{0}$. Let $\mathbf{s}:T_{0}\rightarrow \Tt$ be a splitting. Then we have an isomorphism
\begin{equation*}
    \mu_{n}\times T_{0}\cong \Tt_{0}
\end{equation*}
defined by $(\zeta,t)\mapsto \zeta\cdot \mathbf{s}(t)$. We say a genuine $\Tt$-representation $\pi$ is $\mathbf{s}$-spherical if $\pi$ contains nonzero vectors fixed by $\mathbf{s}(T_{0})$. This notion provides a suitable foundation for a theory of spherical and Iwahori-spherical representations for tame BD covers of connected reductive groups. (See \cite{S04, Gan_Savin}.)

But, if $\gcd(n,p)\neq 1$, it is not obvious how to define genuine ``spherical'' representations of $\Tt$, because the extension \eqref{Eq:CExtTor} need not split over $T_{0}$.

Our main contribution in this work is to define a suitable notion of spherical representations of $\Tt$ beyond the tame case. We call these $\Tt$-representations and the analogous $\Tt_{0}$-representations pseudo-spherical (Definition \ref{Def:PSRep}), following the terminology of \cite{ABPTV, LS10a, LS10b, W09}. This is the first step toward formulating a notion of spherical representations for not necessarily tame BD covers of $p$-adic reductive groups.

For our definition we need to identify a finite-index subgroup $T_{R}\subset T_{0}$ and a splitting $\mathbf{s}_{R}:T_{R}\rightarrow \Tt$ of the extension \eqref{Eq:CExtTor} over $T_{R}$. We construct $\mathbf{s}_{R}$ and prove that $\mathbf{s}_{R}(T_{R})$ is central in $\Tt_{0}$ in Proposition \ref{Prop:TRSplits}. The subgroup $T_{R}$ is defined in terms of $R$, the integral radical of the $n$-th order Hilbert symbol of $F$ (Definition \ref{Def:IntKer}), and is not a congruence subgroup in general.

With $\mathbf{s}_{R}$, we can state our main definition (Definition \ref{Def:PSRep}). 

\begin{Def}
    A genuine representation of $\Tt$ (or $\Tt_{0}$) is pseudo-spherical if it is irreducible and contains nonzero vectors fixed by $\mathbf{s}_{R}(T_{R})$.
\end{Def}
In our main theorem we enumerate and compute the dimension of all of the pseudo-spherical $\Tt_{0}$-representations. With a bit more notation we can state these formulas precisely. 

Let $\Ocal\subset F$ be the ring of integers and $\Ocal^{\times}$ its group of units. Let $\Yt\subset Y$ be the modified cocharacter lattice of $\Tbbt$ (see Section \ref{SS:BD_Covers}).
\begin{Thm}[Theorem \ref{Thm:Main_Thm}]
    Let $\tau$ be a pseudo-spherical representation of $\Tt_{0}$. Then 
    \begin{equation*}
        \dim\tau = \Big|(Y/\Yt)\otimes (\Ocal^{\times}/R)\Big|^{1/2}.
    \end{equation*}
    The number of pseudo-spherical representations is
    \begin{equation*}
        \Big|(\Yt/nY)\otimes (\Ocal^{\times}/R)\Big|.
    \end{equation*}
\end{Thm}

Now assume that $T$ is a maximal split torus in a connected reductive group $G$. Building on the notion of distinguished genuine central characters from \cite{GG18}, which involves the coroot structure of $T\subset G$, we define distinguished pseudo-spherical representations (Definition \ref{Def:Dist_PS_Rep}). Under the mild hypotheses \ref{Hyp_eta_ext} and \ref{Hyp_eta_int} on the BD data we establish existence of distinguished representations and enumerate them, also in Theorem \ref{Thm:Main_Thm}.

Next we will present some evidence that the splitting $\mathbf{s}_{R}$ and our notion of pseudo-spherical representations are broadly valuable. To begin we confine ourselves to the representation theory of $\Tt$. First, our definition is general and uniform; it specializes to the usual notion of spherical in the tame case and the ad hoc notions of pseudo-spherical in the few wild cases appearing in the literature. 

Second, the splitting $\mathbf{s}_{R}$ allows us to prove general results about the representation theory of $\Tt$ and $\Tt_{0}$ beyond the pseudo-spherical representations. We highlight a few examples. In Theorem \ref{Thm:CmptBranch} we prove that the restriction from $\Tt$ to $\Tt_{0}$ is multiplicity-free. This allows us to prove in Corollary \ref{Cor:BK_Type} that every irreducible genuine representation of $\Tt_{0}$ is a type in the sense of Bushnell-Kutzko \cite{BK98}, generalizing \cite[Lemma 3.5]{Wang24} in the tame case. We go on to show that the Hecke algebra attached to each type is (non-canonically) isomorphic to the group algebra $\C[\Yt]$ (Theorem \ref{Thm:Hecke_Alg}). Using $\mathbf{s}_{R}$ we can also compute the center of $\Tt_{0}$ (Corollary \ref{C:centers_are_equal_mod_T_R}), which controls all of the irreducible representations of $\Tt_{0}$ of dimension greater than $1$.

Now we discuss the role of pseudo-spherical representations in the representation theory of $\Gt$, a BD cover of a reductive group $G$. We expect that the distinguished pseudo-spherical representations will provide the torus input for a theory of pseudo-spherical representations of $\Gt$, extending the usual spherical theory in the tame case. One concrete goal toward developing such a theory is to formulate and prove an analogue of the Borel-Casselman theorem for $\Gt$. A Borel-Casselman theorem for non-tame covers would also lay a foundation for generalizing the techniques developed in \cite{GGK24,GGK25,BP25,Zou25} to study Whittaker functionals for tame covers, and is also relevant to understanding the Fourier coefficients of metaplectic Eisenstein series at bad primes \cite{K20}. This goal of an analog of the Borel-Casselman theorem is not purely speculative; in \cite{KT25}, we prove such a theorem for the simply-laced Steinberg double covers over $2$-adic fields using our notion of pseudo-spherical torus representations, vastly generalizing \cite{LS10a, Karasiewicz}.

We conclude the introduction with a summary of the contents of each section. In Section \ref{S:Prelim} we introduce some preliminary material on Hilbert symbols and BD covers and analyze the integral radical $R$. In Section \ref{S:Main_Thm} we construct a splitting of $T_{R}$ in Proposition \ref{Prop:TRSplits} and then use this to define the notion of pseudo-spherical representations (Definitions \ref{Def:PSRep} and \ref{Def:Dist_PS_Rep}). Then we state our main theorem (Theorem \ref{Thm:Main_Thm}) on pseudo-spherical representations. In Sections \ref{S:Torus_Reps_I} we prove our main theorem by using the Heisenberg group structure of $\Tt_{0}/T_{R}$. In Section \ref{S:Torus_Reps_II} we use our notion of pseudo-spherical representations to further develop the representation theory of $\Tt$ and $\Tt_{0}$ proving, among other things, our results on branching (Theorem \ref{Thm:CmptBranch}) and Bushnell-Kutzko types (Corollary \ref{Cor:BK_Type}). In Section \ref{S:Hecke_Alg} we determine the structure of the Hecke algebras attached to genuine irreducible $\Tt_{0}$-representations. Appendix \ref{Appendix} describes the theory of skew-symmetric forms over local rings, which is important for our results in Subsection \ref{SS:Factor_T0}.


\begin{center}
    \normalfont\scshape Acknowledgments
\end{center}


The first author would like to thank the University of Melbourne for its hospitality during his visit in April 2026 and Chenyan Wu for many engaging discussions that helped to catalyze the ideas of this paper.

The second author would also like to thank the University of Melbourne for its hospitality during his visit in August and September 2026, during which the paper was completed, and Chenyan Wu for applying to the University of Melbourne’s Visiting Scholar Program, which made this visit possible. The second author was partially supported by JSPS KAKENHI Grant Number 24K06648.


\section{Preliminaries}\label{S:Prelim}


Let $p\in\Z$ be a prime number. Let $F$ be a finite field extension of the $p$-adic numbers $\Q_{p}$. Let $d = [F:\Q_{p}]$. Let $\Ocal\subset F$ be the ring of integers of $F$ and $\varpi\in F$ a uniformizer. Let $\mu(F)$ be the roots of unity in $F$ and for $k\in\Z_{\geq 1}$ let $\mu_{k}(F)$ be the $k$-th roots of unity in $F$. From now on we suppose that $|\mu_{n}(F)|=n$ and we will just write $\mu_{n} = \mu_{n}(F)$.

Suppose $H$ is a group. Let $H^{\vee}:=\mathrm{Hom}(H,\mathbb{C}^{\times})$ be its group of characters and let $Z(H)$ be its center. If $H$ is an $l$-group in the sense of \cite{BZ76}, we write $\Irr(H)$ for the set of isomorphism classes of smooth irreducible representations of $H$. If $J\subset H$ is a subgroup and $\chi$ is a character of $J$ then we write $\Irr_{\chi}(H)\subset \Irr(H)$ for the subset of $\chi$-isotypic representations. All representations considered in this paper will be smooth representations of $l$-groups.

\subsection{Hilbert symbol and integral radical}

Let 
\[(-,-)=(-,-)_{n}:F^{\times}/F^{\times n}\times F^{\times}/F^{\times n}\rightarrow \mu_{n}\]
be the $n$-th order Hilbert symbol of $F$. See \cite[Chapter V, \S 3]{Neukirch} for details. 

\begin{Lem}\label{PDHilb}
	The group homomorphism $\phi_{F,n}:F^{\times}/F^{\times n}\rightarrow (F^{\times}/F^{\times n})^{\vee}$ defined by $x \mapsto (x,-)_n$ induces an isomorphism of finite abelian groups.
\end{Lem}

\begin{proof}
	The Hilbert symbol on $F^{\times}/F^{\times n}$ is nondegenerate, so $\phi_{F,n}$ is injective. Since $F^{\times}/F^{\times n}$ is a finite abelian group, $\phi_{F,n}$  is also surjective.
\end{proof}

\begin{Lem}\label{Lem:IntTriv}
	There exists $r_{0}\in \mathcal{O}^{\times}$ such that $(r_{0},-)_{n}|_{\mathcal{O}^{\times}}$ is trivial and $(r_{0},\varpi)_{n}$ is a primitive $n$-th root of unity independent of the choice of $\varpi$.
\end{Lem}
\begin{proof}
    The inclusion $\Ocal^{\times} \subset F^{\times}$ induces an exact sequence:
\begin{equation*}
	1\rightarrow \mathcal{O}^{\times}/\mathcal{O}^{\times n}\rightarrow F^{\times}/F^{\times n}\rightarrow F^{\times}/\mathcal{O}^{\times}F^{\times n} \rightarrow 1.
\end{equation*}
By Pontryagin duality we have
\begin{equation*}
	1\rightarrow (F^{\times}/\mathcal{O}^{\times}F^{\times n})^{\vee}\rightarrow (F^{\times}/F^{\times n})^{\vee}\rightarrow (\mathcal{O}^{\times}/\mathcal{O}^{\times n})^{\vee}\rightarrow 1.
\end{equation*}
By the isomorphism $F^{\times}/\mathcal{O}^{\times}F^{\times n}\cong \mathbb{Z}/n\mathbb{Z}$ and Lemma \ref{PDHilb}, there exists $r_{0}\in F^{\times}$ satisfying the conditions of the lemma. To see we may take $r_{0}\in\mathcal{O}^{\times}$, we use the identity $(r_{0},-r_{0})_{n}=1$.
\end{proof}

\begin{Def}\label{Def:IntKer}
	We set
    \[
    R:=\{r\in\Ocal^\times\st \text{$(r,-)_n|_{\Ocal^\times}$ is trivial}\}
    \]
    and call it the {\it integral radical} of the Hilbert symbol.
\end{Def}

\begin{Lem}\label{L:integral_radical_is_cyclic}
The group $R/\mathcal{O}^{\times n}$ is a cyclic group of order $n$ generated by the element $r_{0}$ from Lemma \ref{Lem:IntTriv}.
\end{Lem}
\begin{proof}
    This immediately follows from Lemma \ref{Lem:IntTriv}.
\end{proof}

One can readily see that $R\subset \mathcal{O}^{\times}$ is the subgroup containing $\Ocal^{\times n}$ such that $RF^{\times n}/F^{\times n}$ corresponds to $(F^{\times}/\Ocal^{\times}F^{\times n})^{\vee} \subset (F^{\times}/F^{\times n})^{\vee}\cong F^{\times}/F^{\times n}$. In \cite{KT25} we give an explicit description of $R$ when $n=2$ and $F$ is any finite extension of $\Q_{2}$. When $n=p$ is odd and $F=\Q_{p}(\zeta_{p})$, where $\zeta_{p}$ is a primitive $p$-th root of unity, then $R=(1+p(1-\zeta_{p})\Ocal)\Ocal^{\times p}$ as can be seen from \cite[Exercise 2.13]{CFBook1967}.


\subsection{Decomposition of Hilbert symbol}\label{SS:HilbDecomp}


We can factor $n$ as 
\[
n=p^{\ell}m,
\]
where $\gcd(p, m)=1$. We can write
\[
\mu_n=\mu_{p^{\ell}}\times\mu_m.
\]
Note that the natural surjection $\mu_n\to\mu_{p^{\ell}}$ is given by $\eta\mapsto\eta^m$. Similarly for the surjection $\mu_n\to\mu_m$.

The Hilbert symbol $(-,-)_n$ decomposes into a product of $(-,-)_{p^{\ell}}$ and $(-,-)_m$ as follows. First note that for each $a, b\in F^\times$, we have
\[
(a,b)_{p^{\ell}}=(a, b)_n^m\in\mu_{p^{\ell}}\qand
(a,b)_{m}=(a, b)_n^{p^{\ell}}\in\mu_{m}.
\]
Indeed, this is given by composing with the natural surjections $\mu_n\to\mu_{p^{\ell}}$ and $\mu_n\to\mu_m$. Now, let $A, B\in \Z$ be such that
\[
Ap^{\ell}+Bm=1.
\]
Then
\[
(a, b)_n=(a, b)_n^{Ap^{\ell}+Bm}=(a, b)_{p^{\ell}}^B\cdot (a, b)_m^A.
\]

By restricting the Hilbert symbol $(-,-)_n$ to $\Ocal^\times\times\Ocal^\times$, we obtain
\[
(a, b)_n=(a, b)_{p^{\ell}}^B.
\]
In this way, we view the Hilbert symbol $(-,-)_n$ restricted to $\Ocal^\times\times\Ocal^\times$ as a skew-symmetric bilinear form over $\mu_{p^{\ell}}$. In particular, for all $a, b\in \Ocal^\times$ and $k\in\Z/ p^{\ell}\Z$, we have
\[
(a, b)_n=(b, a)_n^{-1}\qand (a^k, b)_n=(a, b^k)_n=(a, b)_n^k.
\]
The latter is well-defined for $k\mod{p^{\ell}}$ because $(a, b)_{p^{\ell}}^{p^{\ell}}=1$. Also if we fix a primitive $p^{\ell}$-th root of unity $\zeta\in \mu_{p^{\ell}}$, we have the isomorphism
\[
\Z/ p^{\ell}\Z\xrightarrow{\sim}\mu_{p^{\ell}},\quad s\mapsto \zeta^s.
\]
We can then view $(-,-)_n$ as a non-singular skew-symmetric bilinear form 
\begin{equation}\label{E:n-th_Hilbert=p^ell-the_Hilbert}
(-,-)_n=(-,-)_{p^{\ell}}^B:\Ocal^\times/ R\times\Ocal^\times/ R\longrightarrow\mu_{p^{\ell}}.
\end{equation}

In what follows, we decompose this form by viewing $\Ocal^\times/ R$ as a module over the local ring $\Z/ p^{\ell}\Z$. 

\begin{Lem}
    The $\Z/ p^{\ell}\Z$-module $\Ocal^\times/ R$ is free of rank $[F:\Q_p]$.
\end{Lem}
\begin{proof}
    It is well-known that
    \[
    \Ocal^\times\cong \mu(F)\oplus \Z_p^{[F:\Q_p]},
    \]
    where $\mu(F)$ is the set of roots of unity in $F$, which is a cyclic group.
    Hence
    \[
    \Ocal^\times /\Ocal^{\times p^{\ell}}\cong \mu_{p^{\ell}}\oplus (\Z_p/ p^{\ell}\Z_p)^{[F:\Q_p]},
    \]
    noting that $\mu(F)/ \mu(F)^{p^{\ell}}\cong\mu_{p^{\ell}}$. So $\Ocal^\times /\Ocal^{\times p^{\ell}}$ is a free $\Z/p^{\ell}\Z$-module.
    Now,
    \[
    \Ocal^\times/ R\cong (\Ocal^\times/ \Ocal^{\times p^{\ell}})/ (R/\Ocal^{\times p^{\ell}}),
    \]
    and we know that $R/ \Ocal^{\times p^{\ell}}\cong\Z/ p^{\ell}\Z$ by Lemma \ref{L:integral_radical_is_cyclic}, which is a free module of rank 1. (Here, strictly speaking in Lemma \ref{L:integral_radical_is_cyclic} the group $R/ \Ocal^n$ is considered. But one can readily verify the assertion holds for $R/ \Ocal^{p^\ell}$ as well.) 
    
    By selecting a $\Z/p^{\ell}\Z$-basis for $\Ocal^\times /\Ocal^{\times p^{\ell}}$ we may identify $\Ocal^\times /\Ocal^{\times p^{\ell}}$ with the space of column vectors over $\Z/p^{\ell}\Z$, where we write $e_{j}$ for the column vector with $1$ in the $j$-th row and $0$ elsewhere. Let $r^{\prime}\in R/ \Ocal^{\times p^{\ell}}$ be an element of order $p^{\ell}$, and let us view it as a column vector $r^{\prime}\notin \oplus_{j}p\Z/p^{\ell}\Z e_{j}$. Using row operations we can transform the column vector of $r^{\prime}$ into $e_{1}$. So, $r^{\prime}$ can be extended to a $\Z/p^{\ell}\Z$-basis of $\Ocal^\times /\Ocal^{\times p^{\ell}}$ which implies that $(\Ocal^\times/ \Ocal^{\times p^{\ell}})/ (R/\Ocal^{\times p^{\ell}})$ is a free $\Z/p^{\ell}\Z$-module of rank $[F:\Q_{p}]$.
\end{proof}

\begin{Lem}\label{Lem:OR_nonsing}
    The skew-symmetric form $(-,-)_n$ on $\Ocal^\times/ R$ is non-singular in the sense of Appendix \ref{Appendix}.
\end{Lem}
\begin{proof}
We have to show the natural map
\[
\Ocal^\times/ R\longrightarrow\Hom_{\mu_{p^{\ell}}}(\Ocal^\times/ R,\ \mu_{p^{\ell}}),\quad a\mapsto (-, a)_n,
\]
is an isomorphism. By the definition of the integral radical $R$, this map is injective. By the above lemma, one can see that both sides are finite sets with the same cardinality. The lemma follows.
\end{proof}

Lemma \ref{Lem:OR_nonsing} and Theorem \ref{T:decomposition_of_skew-symmetric-module} imply the following. 
\begin{Prop}\label{Prop:Hilb_Decomp}
    Let $\zeta$ be a fixed primitive $p^\ell$-th root of unity. As a skew-symmetric $\Z/ p^{\ell}\Z$-module, $\Ocal^\times/ R$ has the following orthogonal sum decomposition.
    \begin{enumerate}[label=(\alph*), leftmargin=*]
        \item Assume $p\neq 2$. Then
        \[
            \Ocal^\times/ R=\la e_1, f_1\ra\oplus\cdots\oplus\la e_m, f_m\ra,
        \]
        where
        \[
            (e_i, e_i)_n=(f_i, f_i)_n=1\qand  (e_i, f_i)_n=\zeta.
        \]
        \item Assume $p=2$. Then
        \[
            \Ocal^\times/ R=\la u_1\ra\oplus\cdots\oplus\la u_k\ra\;\oplus\;\la e_1, f_1\ra\oplus\cdots\oplus\la e_m, f_m\ra,
        \]
        where the rank 1 module $\la u_i\ra$ can appear only when $\ell=1$, in which case
            \[
            (u_i, u_i)_n=-1,
            \]
        and the 2-dimensional module $\la e_i, f_i\ra$ can appear in any case and we have
            \[
            (e_i, e_i)_n=1,\quad (f_i, f_i)_n=\zeta^s\qand (e_i, f_i)_n=\zeta,
            \]
            where $s\in 2^{\ell-1}\Z/2^\ell\Z$ if $\ell>1$, and $s=0$ if $\ell=1$.
    \end{enumerate}
\end{Prop}

\begin{proof}
    This essentially follows from Theorem \ref{T:decomposition_of_skew-symmetric-module} by specializing to the case where the local ring $A$ is $\Z/ p^\ell\Z$, noting that the maximal ideal of $\mu_{p^{\ell}}\cong\Z/ p^{\ell}\Z$ is $\mu_{p^{\ell}}^{p}\cong p\Z/ p^{\ell}\Z$.

    We do need a couple of extra arguments, though. First, for $(e_i, f_i)_n$, Theorem \ref{T:decomposition_of_skew-symmetric-module} only implies that $(e_i, f_i)_n=\zeta^s$, where $s\equiv 1\pmod{p}$. But then one can replace $f_{i}$ by $f_{i}^{k}$ for a suitable $k$ because $\zeta^s$ is just another primitive $p^\ell$-th root of unity.

    Second, as for (b) with $\ell>1$, Theorem \ref{T:decomposition_of_skew-symmetric-module} only implies $(e_i, e_i)_n=\pm 1$ and $(f_i, f_i)_n=\pm 1$. But if $(e_i, e_i)_n=-1$ and $(f_i, f_i)_n=-1$, then $(e_if_i, e_if_i)_n=1$. So we can take $\{e_if_i, f_i\}$ as a basis. Thus we can assume $(e_i, e_i)_n=1$.
\end{proof}

\quad

In the decomposition of $\Ocal^\times/ R$ in the above proposition, there can be a rank 1 module $\la u_i\ra$ only when $p^\ell=2$, namely $p=2$ and $\ell=1$. In this case, we can be more explicit on the rank 1 modules as follows.

\begin{Prop}\label{P:orthogonal_decomposition_of_quadratic_Hilbert}
Assume $p^\ell=2$. One can choose the number $k$ of the rank 1 modules to be at most 2. 

Furthermore, if we let $S\subset \Ocal^\times/ R$ be the subspace of all isotropic vectors, namely
\[
S=\{u\in\Ocal^\times/ R\st (u, u)_n=1\},
\]
then we have the following.
\begin{enumerate}[label=(\roman*), leftmargin=*]
\item If $-1\in R$, then $S=\Ocal^\times/ R$, so $k=0$.\label{H1}
\item If $(-1, -1)_n=-1$, then $\Ocal^\times/ R=\la-1\ra\oplus S$. \label{H2}
\item If $(-1, -1)_n=1$ and $-1\notin R$, then there exists a rank 2 module $T$ with $-1\in T$ such that $\Ocal^\times/ R=T\oplus T^\perp$ and $T^\perp\subsetneq S$. (Here, we have $T^\perp\subsetneq S$ because $-1\in S$ but $-1\notin T^\perp$.) \label{H3}
\end{enumerate}
\end{Prop}
\begin{proof}
That we can choose $k\leq 2$ is a general fact on the theory of skew-symmetric spaces over $\Z/ 2\Z$, and left to the reader. In what follows, we will show the second part of the proposition. For this purpose, recall that by \eqref{E:n-th_Hilbert=p^ell-the_Hilbert} we have
\[
(a, b)_n=(a, b)_2\quad\text{for all $a, b\in\Ocal^\times$}.
\]

Now, assume $-1\in R$. Then for all $u\in\Ocal^\times$, we have $(u, u)_2=(u, -1)_2=1$. Hence every vector in $\Ocal^\times/ R$ is isotropic. 

Next assume $(-1, -1)_2=-1$. Then the rank 1 module $\la-1\ra$ is non-degenerate. So we have the orthogonal sum decomposition $\Ocal^\times/ R=\la-1\ra\oplus \la-1\ra^\perp$. Since $(u, u)_2=(u, -1)_2$ for all $u\in\Ocal^\times$, we know that $u\in S$ if and only if $u\in \la-1\ra^\perp$.

Finally, assume $(-1, -1)_2=1$ and $-1\notin R$. Since $-1\notin R$, there exists $u_0\in\Ocal^\times$ such that $(-1, u_0)_2=-1$, so $(u_0, u_0)_2=-1$. One can then check that $(-u_0, -u_0)_2=-1$. Since $u_0\neq -u_0$, we have the rank 2 submodule $T:=\{\pm 1, \pm u_0\}$ generated by $u_0$ and $-u_0$. This space is non-degenerate and contains $-1$. We then have the orthogonal sum decomposition $\Ocal^\times/ R=T\oplus T^\perp$. Now, let $u\in T^\perp$. Since $-1\in T$, we have $(u, -1)_2=1$, so $(u, u)_2=1$, namely $u\in S$. 
\end{proof}

\begin{Rmk}
The three cases in Proposition \ref{P:orthogonal_decomposition_of_quadratic_Hilbert} form a partition. Case \ref{H1} occurs if $\sqrt{-1}\in F$; Case \ref{H2} occurs if $[F:\Q_{2}]$ is odd (so that necessarily $\sqrt{-1}\notin F$) because then $(-1, -1)_2=(-1, N_{F/\Q_2}(-1))_{\Q_2}=(-1, -1)_{\Q_2}=-1$, where $(-, -)_{\Q_2}$ is the quadratic Hilbert symbol for $\Q_2$; Case \ref{H3} occurs if $\sqrt{-1}\notin F$ and $[F:\Q_{2}]$ is even, because then $(-1, -1)_2=(-1, N_{F/\Q_2}(-1))_{\Q_2}=1$ and $-1=1-2\in 1+2\Ocal^\times$ so that $-1\notin R$ by \cite[Proposition A.7]{KT25}.
\end{Rmk}


\subsection{Reductive groups}


Let $\mathbb{G}$ be a connected split reductive linear algebraic group defined over $F$. Let $\mathbb{T}\subset \mathbb{G}$ be a maximal split torus with root datum $(X,\Phi,Y,\Phi^{\vee})$, where $X$ is the character lattice, $Y$ the cocharacter lattice, $\Phi\subset X$ the set of roots, and $\Phi^{\vee}\subset Y$ the set of coroots. Let $Y^{sc} = \mathrm{span}_{\Z}(\Phi^{\vee})$. For $x\in X$ and $t\in \Tbb$, we will sometimes write $t^{x}$ for $x(t)$. Let $N(\Tbb)$ be the normalizer of $\Tbb$ in $\Gbb$ and let $W$ be the Weyl group of $(\mathbb{G},\mathbb{T})$. For each $\alpha\in \Phi$ let $\mathbb{U}_{\alpha}$ be the associated root subgroup. Let $\mathbb{B}=\mathbb{B}^{+}\supset \Tbb$ be a Borel subgroup. Let $\mathbb{B}^{-}$ be the opposite Borel subgroup and let $\mathbb{U}^{\pm}$ be the unipotent radical of $\mathbb{B}^{\pm}$. Let $\Phi^{+}$ be the set of positive roots of $\Tbb$ relative to $\mathbb{B}$ and $\Phi^{-}=-\Phi^{+}$. Let $\Delta\subset \Phi^{+}$ be the set of simple roots associated to $\Bbbb$. We fix a Chevalley pinning with respect to $\mathbb{B}$, which for each root $\alpha\in\Phi$ gives a morphism of algebraic groups
\begin{equation*}
    x_{\alpha}:\Gbb_{a}\rightarrow \mathbb{U}_{\alpha}.
\end{equation*}
With this we can define some one-parameter subgroups for each $\alpha\in\Phi$,
\begin{align*}
    w_{\alpha}:\Gbb_{m}\rightarrow N(\Tbb),& \hspace{1cm} a\mapsto x_{\alpha}(a)x_{-\alpha}(-a^{-1})x_{\alpha}(a);\\
    h_{\alpha}:\Gbb_{m}\rightarrow \Tbb\phantom{N()},& \hspace{1cm} a\mapsto w_{\alpha}(a)w_{\alpha}(-1).
\end{align*}

Let $G := \Gbb(F)$ be the group of $F$ points of $\Gbb$. Similarly, if $\mathbb{H}$ is some algebraic subgroup of $\Gbb$ defined over $F$ we write $H:=\mathbb{H}(F)$ for its group of $F$-points. We abuse notation and continue to write $x_{\alpha}$, $w_{\alpha}$, $h_{\alpha}$ for the induced maps on $F$-points.

\subsection{Brylinski-Deligne covers}\label{SS:BD_Covers}

Let $\Gbbt$ be a Brylinski-Deligne (BD) cover of $\mathbb{G}$ incarnated by the datum $(D,\eta)$, where $D:Y\otimes Y\rightarrow \mathbb{Z}$ is a fair bisector (\cite[Definitions 2.2 and 2.11]{W14}) and $\eta:Y^{sc}\rightarrow F^{\times}$ is a group homomorphism. The BD cover $\Gbbt$ is an object in the category of sheaves of groups on the big Zariski site of $\mathrm{Spec}(F)$, and it fits into a central extension in this category
\begin{equation}\label{Eq:CExt}
    \mathbb{K}_{2}\hookrightarrow \overline{{\mathbb{G}}}\stackrel{\pr}{\twoheadrightarrow} \mathbb{G},
\end{equation}
where $\mathbb{K}_{2}$ is Quillen's $K_{2}$ group.
For details on BD covers see \cite{BD01,W16,W18,GG18}; for a concise discussion of the $(D,\eta)$ incarnation specifically, see \cite[Sections 2.6, 2.7]{GG18}.

Recall that the (not necessarily symmetric) bilinear form $D$ can be used to define a $W$-invariant quadratic form $Q:Y\rightarrow \Z$ defined by
\begin{equation*}
    Q(y): = D(y,y).
\end{equation*}
Let $B=B_{Q}:Y\times Y\rightarrow \Z$ be the symmetric bilinear form attached to $Q$. Note that for $y_{1},y_{2}\in Y$ we have 
\begin{equation*}
    B(y_{1},y_{2}) = D(y_{1},y_{2})+D(y_{2},y_{1}).
\end{equation*}

For any algebraic subgroup $\mathbb{H}\subset \Gbb$ we get a BD cover $\overline{\mathbb{H}}$ by pulling back the central extension \eqref{Eq:CExt} along the inclusion $\mathbb{H}\subset \Gbb$. From \cite[Proposition 3.2 and subsequent comment]{BD01}, any unipotent subgroup $\mathbb{H}\subset \Gbb$ admits a unique splitting $\mathbf{s}_{\mathbb{H}}:\mathbb{H}\rightarrow \overline{\mathbb{H}}$ of the central extension
\begin{equation*}
     \mathbb{K}_{2}\hookrightarrow \overline{{\mathbb{H}}}\stackrel{\pr}{\twoheadrightarrow} \mathbb{H}.
\end{equation*}
Thus for each $\alpha\in\Phi$ we get a unique splitting $\mathbf{s}_{\alpha}:=\mathbf{s}_{\mathbb{U}_{\alpha}}:\mathbb{U}_{\alpha}\rightarrow \overline{\mathbb{U}}_{\alpha}$. Let $\xt_{\alpha}:=\mathbf{s}_{\alpha} \circ x_{\alpha}$. With this we can define
\begin{align*}
    \wt_{\alpha}:\Gbb_{m}\rightarrow \overline{N(\Tbb)},& \hspace{1cm} a\mapsto \xt_{\alpha}(a)\xt_{-\alpha}(-a^{-1})\xt_{\alpha}(a);\\
    \htt_{\alpha}:\Gbb_{m}\rightarrow \Tbbt\phantom{N()},& \hspace{1cm} a\mapsto \wt_{\alpha}(a)\wt_{\alpha}(-1).
\end{align*}

Note that if $y\in Y$, then $D(y,-)\in X$. The incarnation of $\Gbbt$ using $(D,\eta)$ implies that $\Tbb$ admits a section 
\[
\mathbf{s}:\mathbb{T}\rightarrow \Tbbt
\]
such that  $\mathbf{s}(id_{\Tbb})=id_{\Tbbt}$, and for all $y\in Y$, $u\in \mathbb{G}_{m}$, and $t\in\Tbb$
\begin{equation}\label{Eq:TSectionMult}
    \mathbf{s}(y(u))\cdot \mathbf{s}(t) = \{u,t^{D(y,-)}\}\;\mathbf{s}(y(u)\cdot t).
\end{equation}
(See \cite[Section 0.N.5]{BD01} for $\{-,-\}$ notation.)

We note that for $\alpha\in\Phi$ the section $\mathbf{s}\circ h_{\alpha}$ need not equal $\htt_{\alpha}$. The discrepancy for simple roots is described in \cite[Proposition 2.5]{GG18}, which states that for $\alpha\in \Delta$ and $u\in \mathbb{G}_{m}$
\begin{equation}\label{Eqn:SimpleCoRootSection}
    \mathbf{s}(\alpha^{\vee}(u))\{\eta(\alpha^{\vee}),u\} = \htt_{\alpha}(u).
\end{equation}
For $\alpha\in\Phi$, by \cite[Propositions 11.3 and 11.9]{BD01} we have
\begin{equation}\label{Eq:TSectionWConj}
    \wt_{\alpha}(1)\cdot \mathbf{s}(y(u))\cdot \wt_{\alpha}(-1) = \mathbf{s}(y(u))\cdot \htt_{\alpha}(u^{-\la\alpha,y\ra}).
\end{equation}

The pair $(\Gbbt,n)$ has attached to it a root datum
\begin{equation*}
    (\Xt,\Phi_{Q,n},\Yt,\Phi_{Q,n}^{\vee}).
\end{equation*}
For details see \cite{Mc12,W14,W18}. We only need $\Yt$ and $\Phi_{Q,n}^{\vee}$, where
\begin{equation*}
    \Yt:=\{y\in Y\;|\;B_{Q}(y,-):Y\rightarrow n\Z\},
\end{equation*}
and
\[
\Phi_{Q, n}^{\vee}:=\{n_\alpha\alpha^\vee\st \alpha\in\Phi\},\qquad
n_{\alpha} := \frac{n}{\gcd(n,Q(\alpha^{\vee}))}.
\]
Let
\[
\Yt^{sc}:={\Span}_{\Z}\{\Phi_{Q,n}^{\vee}\}.
\]

Now we define the main $p$-adic groups that we will be working with. Let $\Gt$ be defined by first taking $F$-points of the sequence \eqref{Eq:CExt} and then pushing out by the $n$-th order Hilbert symbol $\mathbb{K}_{2}(F)\rightarrow \mu_{n}$ to get a topological central extension
\begin{equation}\label{Eq:CExt_F}
    \mu_{n}\hookrightarrow \Gt \stackrel{\pr}{\twoheadrightarrow} G.
\end{equation}
Similarly, for any algebraic subgroup $\mathbb{H}\subset \Gbb$ we can define $\overline{H}$ by this procedure. Alternatively, if $H\subset G$, we write $\overline{H}:=\pr^{-1}(H)$. When $H$ is the group of $F$-points of an algebraic subgroup of $\Gbb$, these two constructions agree. We note that $\overline{H}$ is an $l$-group in the sense of \cite{BZ76}. All representations of covering groups will be $\epsilon$-genuine, meaning $\mu_{n}$ acts through a fixed faithful character $\epsilon:\mu_{n}\rightarrow \C^{\times}$.

As in the case of $G$ we abuse notation and continue to write $\xt_{\alpha}$ for the map on $F$-points induced by $\xt_{\alpha}:\Gbb_{a}\rightarrow \overline{\mathbb{U}}_{\alpha}$. Similarly, for $\wt_{\alpha}$ and $\htt_{\alpha}$. The identities above in $\Tbbt$ or $\Gbbt$ descend to identities in $\Tt$ and $\Gt$, respectively. For example, the identity \eqref{Eq:TSectionMult} in $\Tbbt$ becomes in $\Tt$
\begin{equation}\label{Eq:TSectionMult_F}
    \mathbf{s}(y(u))\cdot \mathbf{s}(t) = (u,t^{D(y,-)})_{n}\cdot\mathbf{s}(y(u)\cdot t).
\end{equation}
If $t\in T$, we will sometimes write $\tilde{t}=\mathbf{s}(t)$. For $y\in Y$ and $u\in F^\times$, we write
\[
\yt(u)=\s(y(u)).
\]
We will write $\tbb\in \Tt$ for an arbitrary lift of $t$ to $\Tt$. We will only use the notation $\tbb$ when the choice of lift does not matter.

For any $y,z\in Y$ and $u,v\in F^{\times}$, in $\Tt$ we have the commutator identity
\begin{equation}\label{Comm_Id}
	[\ybb(u),\zbb(v)]:= \ybb(u)\zbb(v)\ybb(u)^{-1}\zbb(v)^{-1} = (u,v)^{B_{Q}(y,z)}\in \mu_{n}.
\end{equation}

\subsection{Distinguished characters}\label{SS:Dist_Char}

In this section we review the notion of distinguished characters of the center of the torus $Z(\Tt)$. For additional details see \cite[\S 6 and 7]{GG18}. We emphasize that this notion depends on a coroot structure and so involves $\Gbbt$, and not just $\Tbbt$. We will make this precise after some preliminaries.

From the lattice $\Yt$ we get a split torus $\Tbb_{Q,n}$ over $F$. From the inclusion $\Yt\hookrightarrow Y$ we get an isogeny of split tori 
\[
h:\Tbb_{Q,n}\rightarrow \Tbb.
\]
Now we may pull back the $\mathbb{K}_{2}$-extension of $\Tbb$ along $h$ to get an extension
\begin{equation*}
    \mathbb{K}_{2}\hookrightarrow \Tbbt_{Q,n}\twoheadrightarrow \Tbb_{Q,n}.
\end{equation*}
This extension canonically splits over $\ker h$. So, we have a map $\ker h\hookrightarrow \Tbbt_{Q,n}$.

Taking $F$-points and pushing out along the $n$-th order Hilbert symbol $\mathbb{K}_{2}(F)\rightarrow \mu_{n}$ gives a central extension of groups
\begin{equation}\label{Eq:CExt_TQn}
    \mu_{n}\hookrightarrow \Tt_{Q,n}\stackrel{\pr_{Q,n}}{\twoheadrightarrow} T_{Q,n},
\end{equation}
where $T_{Q,n}:=\Tbb_{Q,n}(F)$. We note that $\Tt_{Q,n}$ is an abelian group generated by $\mu_n$ and the elements of the form $\yh(u)$, where $y\in\Yt$ and $u\in F^{\times}$, subject to the relations:
\begin{itemize}
    \item $\mu_{n}\subset \Tt_{Q,n}$ is central;
    \item $\yh_{1}(u)\yh_{2}(u) = (u,u)_n^{D(y_{1},y_{2})}\widehat{(y_{1}+y_{2})}(u)$;
    \item $[\yh_{1}(u_{1}),\yh_{2}(u_{2})]=1$ for all $y_{j}\in \Yt$ and $u_{j}\in F^{\times}$;
    \item $\yh(u)\yh(v) = (u,v)_n^{Q(y)}\yh(uv)$.
\end{itemize}

From \cite{W09} we know that the map $h$ induces an exact sequence
\begin{equation}\label{E:T_Q_n_and_Z(T)}
    1\rightarrow [\ker h](F)\rightarrow \Tt_{Q,n}\rightarrow Z(\Tt)\rightarrow 1,
\end{equation}
where the map $\Tt_{Q,n}\rightarrow Z(\Tt)$ is given by $\yh(u)\mapsto \mathbf{s}(y(u))$.

The inclusion $nY\hookrightarrow \Yt$ defines an isogeny 
\[
g:\Tbb\rightarrow \Tbb_{Q,n}.
\]
By \cite[Lemma 3.1]{GG18}, we have $[\ker h](F)\subset g(T)$. We can construct a lift 
\[
\gt:T\rightarrow \Tt_{Q,n},\qquad y(u) \mapsto \widehat{ny}(u).
\]
With this we have $[\ker h](F)\subset \gt(T)\subset \Tt_{Q,n}$.

The inclusion $\Yt^{sc}\hookrightarrow \Yt$ induces an isogeny $T_{Q,n}^{sc}\rightarrow T_{Q,n}$. From \cite[pg 219-220]{GG18}, since $D$ is a fair bisector, there is a well-defined splitting 
\begin{equation}\label{E:s_eta_definition}
\mathbf{s}_{\eta}:T_{Q,n}^{sc}\longrightarrow \Tt_{Q,n},\quad y(u)\mapsto \yh(u)\cdot(\eta(y),u)_n.
\end{equation}

The above discussion is summarized in the following commutative diagram:
\begin{equation*}
\begin{tikzcd}
     &\Tt_{Q, n} \ar[d] & \\
    T_{Q, n}^{sc}\ar[ru, "{\s_\eta}"]\ar[r]&T_{Q, n}  & T \ar[l, "{g}"] \ar[lu, "{\gt}"']
\end{tikzcd}
\end{equation*}
Note that the bottom horizontal arrows are isogenies.

\begin{Def}[\cite{GG18} Section 6.4]\label{Def:DistGenChar}
    A genuine character $\chi:\Tt_{Q,n}\rightarrow \C^{\times}$ is {\it distinguished} if it is trivial on both $\mathbf{s}_{\eta}(T_{Q,n}^{sc})$ and $\gt(T)$.
\end{Def}

\begin{Rmk}The notion of distinguished characters emerges from the study of $L$-groups of BD covers. This connection is thoroughly explained in \cite[Section 6]{GG18}.
\end{Rmk}

Because $\gt(T)\supset [\ker h](F)$, a distinguished character $\chi$ of $\Tt_{Q,n}$ descends to a genuine character of $Z(\Tt)$ by \eqref{E:T_Q_n_and_Z(T)}. We will also refer to $\chi$ as a distinguished genuine character of $Z(\Tt)$. We note that distinguished characters of $Z(\Tt)$ are Weyl-invariant by \cite[Theorem 6.8]{GG18}.

Distinguished characters do not always exist; in the next proposition we recall \cite[Theorem 6.6]{GG18}, which gives a sufficient condition for their existence and a description of a torsor structure when they exist.

Let 
\[
J=J_n:= nY+\Yt^{sc}
\]
and let
\[
\eta_{n}:Y^{sc} \rightarrow F^{\times}/F^{\times n}
\]
be the composition of $\eta$ with the natural quotient map $F^{\times}\rightarrow F^{\times}/F^{\times n}$. Let $\Irr_{\epsilon}^{dist}(\Tt_{Q,n})$ be the set of distinguished  characters. We will write $\Irr_{\epsilon}^{dist}(Z(\Tt))$ for this set of distinguished characters viewed as characters of $Z(\Tt)$.

\begin{Prop}[\cite{GG18} Theorem 6.6]\label{Prop:GG_6.6}
    Assume that $\eta_{n}$ can be extended to a homomorphism of $Y$. Then:
    \begin{enumerate}[leftmargin={2em}]
        \item The set $\Irr_{\epsilon}^{dist}(\Tt_{Q,n})$ is nonempty.
        \item The set $\Irr_{\epsilon}^{dist}(\Tt_{Q,n})$ is a torsor under $\mathrm{Hom}((\Yt/J)\otimes F^{\times},\C^{\times})$.
    \end{enumerate}
\end{Prop}

Now we introduce some notation. For $L\subset Y$ a sublattice let
\[
T^L:=\text{image of the isogeny $L\otimes F^\times\to Y\otimes F^\times=T$},
\]
and if $L\subset Y_{Q, n}$ we let
\[
T_{Q, n}^L:=\text{image of the isogeny $L\otimes F^\times\to Y_{Q, n}\otimes F^\times=T_{Q, n}$}.
\]
We are particularly interested in the case where $L=J$ and $nY$.

\begin{Lem}\label{L:existence_of_s_J}
    Assume $\Tt_{Q, n}$ admits a distinguished character. Then there exist splittings $T_{Q, n}^J\to \Tt_{Q, n}$ and $T^J\to Z(\Tt)$ so that the diagram
    \[
    \begin{tikzcd}
     \Tt_{Q, n} \ar[r] & Z(\Tt)\\
    T_{Q, n}^{J}\ar[u]\ar[r]&T^J \ar[u]
    \end{tikzcd}
    \]
    commutes.
\end{Lem}
\begin{proof}
    Let $\chi$ be a distinguished character of $\Tt_{Q, n}$. Since $\chi$ is genuine, the map $\pr_{Q,n}|_{\ker\chi}$ is injective and thus defines an isomorphism onto its image. Thus, \eqref{Eq:CExt_TQn} splits over the subgroup $\pr_{Q,n}(\ker\chi)\subset T_{Q,n}$. So the extension splits over any subgroup of $\pr_{Q,n}(\ker\chi)$.

    By definition $\mathbf{s}_{\eta}(T_{Q,n}^{sc})$ and $\gt(T)$ are contained in the kernel of $\chi$. So, $T_{Q,n}^{J}\subset \pr_{Q,n}(\ker\chi)$, and thus the extension splits over $T_{Q,n}^{J}$.

    Since $\chi$ descends to $Z(\Tt)$, one can also ``descend the splitting" to a splitting of $T^J$.
\end{proof}

We denote the splittings of the lemma by
\begin{equation}\label{E:section_S_J}
\s_J: T_{Q, n}^J\longrightarrow \Tt_{Q,n}\qand \s_J:T^J\longrightarrow Z(\Tt),
\end{equation}
using the same symbol.

Note that the splitting $\s_J:T_{Q, n}^{J}\to \Tt_{Q, n}$ can be considered as an extension of both $\s_\eta$ and $\s_n$ in the following sense. For each $y\in \Yt^{sc}$ and $u\in F^\times$, we have
\[
\s_\eta(y(u))=\s_J(y(u)),
\]
where, strictly speaking, on the left-hand side $y(u)=y\otimes u\in Y_{Q,n}^{sc}\otimes F^\times=T_{Q,n}^{sc}$ and on the right-hand side $y(u)=y\otimes u\in \Yt\otimes F^\times=T_{Q,n}\supseteq T_{Q,n}^{J}$. Also, consider the splitting
\[
\mathbf{s}_{n}:T_{Q,n}^{\,nY}\rightarrow \Tt_{Q,n},\qquad ny(u)\mapsto \widehat{ny}(u).
\]
Note that the images of $\mathbf{s}_{n}$ and $\tilde{g}$ agree. Then for all $ny\in nY$ and $u\in F^\times$, we have
\[
\s_{n}(ny(u))=\s_J(ny(u)).
\]
Hence we say that the splitting $\s_J$ extends both of the splittings $\s_\eta$ and $\s_n$.

\subsection{Distinguished characters of $Z(\Tt_0)$.}

Analogously to distinguished characters of $Z(\Tt)$ (or equivalently of $\Tt_{Q, n}$), in this subsection, we define distinguished characters of $Z(\Tt_0)$, where we recall $\Tt_0$ is the preimage of the integral points 
\[
T_0=Y\otimes \Ocal^\times \subset T.
\]

Let
\[
\Tt_{Q,n}(\Ocal):=\{\epsilon\,\yh(u)\in \Tt_{Q,n}\st u\in\Ocal^\times, \;\epsilon\in\mu_n\}.
\]
We then have the commutative diagram
\begin{equation*}
\begin{tikzcd}
     &\Tt_{Q, n}(\Ocal) \ar[d] & \\
    T_{Q, n}^{sc}(\Ocal)\ar[ru, "{\s_\eta}"]\ar[r]&T_{Q, n}(\Ocal)  & T_0 \ar[l, "{g}"] \ar[lu, "{\gt}"']\rlap{\; ,}
\end{tikzcd}
\end{equation*}
where
\[
T_{Q, n}^{sc}(\Ocal)=Y_{Q,n}^{sc}\otimes\Ocal^\times\qand T_{Q,n}(\Ocal)=Y_{Q,n}\otimes\Ocal^\times.
\]
Note that
\[
T_{Q,n}^{sc}(\Ocal)\subset T_{Q,n}^{sc}\qand T_{Q,n}(\Ocal)\subset T_{Q,n}.
\]

Recall the image of the natural map $\Tt_{Q,n}\to \Tt$ is the center $Z(\Tt)$, so we have $\Tt_{Q, n}\twoheadrightarrow Z(\Tt)$, and the kernel of it is contained in $\gt(T)$, so a distinguished character of $\Tt_{Q, n}$ descends to $Z(\Tt)$. Hence a distinguished character is also viewed as a character of $Z(\Tt)$.

For the integral points, we certainly have the natural map $\Tt_{Q, n}(\Ocal)\to \Tt$ simply by restricting $\Tt_{Q,n}\to \Tt$. However, the image is not all of $Z(\Tt_0)$. So we have the natural map
\[
\Tt_{Q, n}(\Ocal)\longrightarrow Z(\Tt_0),
\]
which is not necessarily surjective.

With this said, we make the following definition.
\begin{Def}\label{Def:DistGenChar}
    A genuine character $\chi:Z(\Tt_0)\rightarrow\C^\times$ is {\it distinguished} if its pullback to $\Tt_{Q, n}(\Ocal)$ via the natural map $\Tt_{Q, n}(\Ocal)\longrightarrow Z(\Tt_0)$ is trivial on both $\mathbf{s}_{\eta}(T_{Q,n}^{sc}(\Ocal))$ and $\gt(T_0)$.
\end{Def}

Let us mention
\begin{Prop}\label{Prop:dist_cmpt_dist}
    If $Z(\Tt)$ admits a distinguished character, then so does $Z(\Tt_{0})$.
\end{Prop}
\begin{proof}
    Let $\chi:\Tt_{Q,n}\to\C^\times$ be a distinguished character of $\Tt_{Q,n}$. The restriction $\chi|_{\Tt_{Q, n}(\Ocal)}$ is trivial on  both $\mathbf{s}_{\eta}(T_{Q,n}^{sc}(\Ocal))$ and $\gt(T_0)$. Descend it to a character of the image of $\Tt_{Q, n}(\Ocal)\to Z(\Tt_0)$, and extend it to all of $Z(\Tt_0)$. The resulting character is a distinguished character of $Z(\Tt_0)$.
\end{proof}


\section{Main Theorem}\label{S:Main_Thm}


In this section we define the notion of pseudo-spherical representations and distinguished pseudo-spherical representations, and state our main theorem. These notions and the subsequent theorems all emerge from isolating the following subgroup of $T_{0}$. Let $T_{R}$ be the subgroup of $T_{0}$ generated by elements of the form $y(r)$, where $y\in Y$ and $r\in R$.

\subsection{Splitting of $T_R$}

To define the notion of pseudo-spherical representations, let
\[
\mathbf{s}_{R}:=\mathbf{s}|_{T_{R}}:T_{R}\longrightarrow \Tt_{0}
\]
be the restriction of the section $\mathbf{s}:T \rightarrow \Tt$ (Subsection \ref{SS:BD_Covers}). Then we have

\begin{Prop}\label{Prop:TRSplits}
The map $\s_R$ has the following properties:
\begin{enumerate}[label=$\bullet$, leftmargin=1.5em]
    \item It is a group homomorphism. 
    \item The image of $\s_R$ is central in $\Tt_0$, namely $\s_R(T_R)\subset Z(\Tt_0)$.
    \item If $T\subset G$ is a maximal torus and $\eta$ is valued in $\Ocal^{\times}F^{\times n}$, then $\mathbf{s}_{R}$ is $W$-equivariant. 
\end{enumerate}
\end{Prop}
\begin{proof}
    The group $T_{R}$ is generated by elements of the form $y(r)$, where $y\in Y$ and $r\in R$. Let $y\in Y$, $r\in R$ and $t\in T_{R}$. Then by \eqref{Eq:TSectionMult_F}
    \begin{equation*}
        \mathbf{s}_{R}(y(r))\mathbf{s}_{R}(t) = (r,t^{D(y,-)})_n\mathbf{s}_{R}(y(r)t)
        =\mathbf{s}_{R}
        (y(r)t),
    \end{equation*}
    where the second equality follows because $r\in R$ and $t^{D(y,-)}\in \Ocal^\times$. Thus $\mathbf{s}_{R}$ is a group homomorphism. The identity \eqref{Comm_Id} implies $\mathbf{s}_{R}(T_{R})\subset Z(\Tt_{0})$. The Weyl invariance follows from \eqref{Eqn:SimpleCoRootSection} and \eqref{Eq:TSectionWConj}, by using the assumption that $\eta$ is valued in $\Ocal^{\times}F^{\times n}$.
\end{proof}

Via this splitting, we often view $T_R$ as a subgroup of $\Tt$ and denote $\s_R(T_R)$ simply by $T_R$ when there is no danger of confusion.

\begin{Rmk}
    Though $T_R\subset Z(\Tt_{0})$ as above, in general $T_R\nsubset Z(\Tt)$ and hence $Z(\Tt_0)\nsubset Z(\Tt)$. This is because $\pr(Z(\Tt))=T^{Y_{Q, n}}$, and hence we can have $T_R\nsubset Z(\Tt)$ because in generally $Y_{Q, n}\neq Y$.
\end{Rmk}

\subsection{Definition of (distinguished) pseudo-spherical representations}
Now we can define two important families of representations of $\Tt_{0}$ as follows.

\begin{Def}\label{Def:PSRep}
    An irreducible representation $\pi$ of $\Tt$ (or $\Tt_{0}$) is called {\it pseudo-spherical} or {\it $\s_R$-pseudo-spherical} if $\pi$ is genuine and the subspace of $\mathbf{s}_R(T_{R})$-fixed vectors is nonzero.
\end{Def}

\begin{Rmk}
    For the above definition we use the splitting $\s_R:T_R\to\Tt$. But there are as many (continuous) splittings as the number of characters $T_R\to\mu_n$. Hence for each splitting $\s_R':T_R\to \Tt$, one can define the notion of $\s_R'$-pseudo-spherical representations. However in this paper, we only use the splitting $\s_R:T_R\to \Tt$.
\end{Rmk}

If $T$ is a split maximal torus in a connected reductive group $G$, we can identify a distinguished subset of pseudo-spherical representations as follows.

\begin{Def}\label{Def:Dist_PS_Rep}
    An irreducible representation $\pi$ of $\Tt$ (or $\Tt_{0}$) is called a {\it distinguished pseudo-spherical} representation if $\pi$ is a pseudo-spherical representation with a distinguished central character.
\end{Def}

Let 
\[
\Irr_{\epsilon}^{ps}(\Tt_{0})
\]
be the set of isomorphism classes of pseudo-spherical representations, and let 
\[
\Irr_{\epsilon}^{dps}(\Tt_{0}) \subset \Irr_{\epsilon}^{ps}(\Tt_{0})
\]
be the subset of distinguished pseudo-spherical representations.

\subsection{Main theorem}
To state our main theorem, we need to record the following two hypotheses about $\eta: Y^{sc}\to F^\times$, which guarantee existence of distinguished characters with good properties.

\begin{enumerate}[label=\textbf{H\arabic*)}, ref=H\arabic*)]
    \item\label{Hyp_eta_ext} $\eta_{n}$ can be extended to a homomorphism of $Y$. 
    \item\label{Hyp_eta_int} $\eta$ is valued in $\Ocal^{\times}F^{\times n}$.
\end{enumerate}

\quad

We can now state our main theorem.
\begin{Thm}\label{Thm:Main_Thm} $ $
    \begin{enumerate}[label=$\bullet$, leftmargin=*]
        \item Let $\tau$ be a pseudo-spherical representation of $\Tt_0$. Then
        \begin{equation*}
        \operatorname{dim}\tau = \Big|Y/ Y_{Q, p^{\ell}}\Big|^{\frac{[F:\Q_p]}{2}}=\Big|(Y/\Yt)\otimes (\Ocal^{\times}/R)\Big|^{1/2},
    \end{equation*}
and
    \begin{equation*}
        \Big|\Irr_{\epsilon}^{ps}(\Tt_{0})\Big|=\Big|Y_{Q,p^{\ell}}/ p^{\ell}Y\Big|^{[F:\Q_p]}=\Big|(\Yt/nY)\otimes (\Ocal^{\times}/R)\Big|.
    \end{equation*}
    
    \item Suppose that $T\subset G$ is a maximal split torus and assume \ref{Hyp_eta_ext} and \ref{Hyp_eta_int}. For $k\in \mathbb{Z}_{\geq 1}$, let $J_{k}=kY+Y^{sc}_{Q,k}$. Then
    \begin{equation*}
        \Big|\Irr_{\epsilon}^{dps}(\Tt_{0})\Big|=\Big|Y_{Q,p^{\ell}}/ J_{p^{\ell}}\Big|^{[F:\Q_p]}=\Big|(\Yt/J_{n})\otimes (\Ocal^{\times}/R)\Big|.
    \end{equation*}
    Moreover, each distinguished pseudo-spherical representation of $\Tt_0$ is $W$-invariant.
    \end{enumerate}
\end{Thm}

\begin{Rmks}
\quad
\begin{enumerate}[label=(\arabic*), leftmargin=*]
\item The statement of the theorem is a bit redundant in that the count in the first statement can be recovered from the second by considering $\Tbb=\Gbb$, in which case $Y^{sc}=0$ and $J_n=nY$.
\item Some restrictions on $\eta$ are necessary because distinguished characters do not exist in general. By Proposition \ref{Prop:GG_6.6},  the hypothesis \ref{Hyp_eta_ext} implies the existence of distinguished characters. 
\item Adding the second hypothesis \ref{Hyp_eta_int} essentially guarantees that the two splittings $\s_J$ and $\mathbf{s}_{R}$ agree on the overlap $T^J\cap T_R$, as we will see in Subsection \ref{SS:Compat_Split}.
\item Both \ref{Hyp_eta_ext} and \ref{Hyp_eta_int} are satisfied when $\eta$ is trivial and thus the set of covers satisfying both hypotheses is nonempty. Furthermore, we note that the representation theory of a BD cover with data $(D,\eta)$ can be related using $z$-extensions to the representation theory of a BD cover with data $(D^{\prime}, \one)$. For details, see \cite[Sections 2.9 and 11.4]{GG18}.
\end{enumerate}
\end{Rmks}
\quad

\subsection{Compatibility of $\s_J$ and $\s_R$}\label{SS:Compat_Split}
The hypotheses \ref{Hyp_eta_ext} and \ref{Hyp_eta_int} guarantee that the two splittings $\s_J$ and $\s_R$ agree on the overlap $T^J\cap T_R$.

\begin{Lem}\label{L:compatibility_of_S_J_and_S_R}
    Assume hypotheses \ref{Hyp_eta_ext} and \ref{Hyp_eta_int}. Then
    \[
    \s_J|_{T^J\cap T_R}=\s_R|_{T^J\cap T_R};
    \]
    namely the two splittings $\s_J$ and $\s_R$ agree on the overlap.
\end{Lem}
\begin{proof}
    It suffices to check this on $T_{\Yt^{sc}}\cap T_{R}$. Using Smith normal form we can construct $\{y_{j}\}$ a basis for $Y$ such that $\{d_{j}y_{j}|d_{j}\neq 0\}$ is a basis for $\Yt^{sc}$. Suppose that $\prod_{j}(d_{j}y_{j})(u_{j})\in T_{\Yt^{sc}}\cap T_{R}$, where $u_{j}\in F^{\times}$. Since $\prod_{j}(d_{j}y_{j})(u_{j})\in T_{R}$, we have $u_{j}^{d_{j}}\in R$ for all $j$.

    Now by the definition of $\mathbf{s}_{\eta}$ and $\s_R=\s|_{T_R}$, we have
    \begin{align*}
        \mathbf{s}_{\eta}(\prod_{j}(d_{j}y_{j})(u_{j})) 
        &= \prod_{j}(\eta_{n}(d_{j}y_{j}),u_{j})_n\;\s_R((d_{j}y_{j})(u_{j}))\qquad(\text{definition of $\s_\eta$})\\
        &= \prod_{j}(\eta_{n}(y_{j}),u_{j}^{d_{j}})_n\;\s_R(y_{j}(u_{j}^{d_{j}}))\qquad(\text{hypothesis \ref{Hyp_eta_ext}})\\
        &=\prod_{j}\mathbf{s}(y_{j}(u_{j}^{d_{j}}))\qquad (\text{by $u_{j}^{d}\in R$ and hypothesis \ref{Hyp_eta_int}}) \\
        &=\s_R(\prod_{j}y_{j}(u_{j}^{d_{j}}))\\
        &=\s_R(\prod_{j}(d_{j}y_{j})(u_{j})),
    \end{align*}
where the third equality uses hypothesis \ref{Hyp_eta_int}, so that $(\eta_{n}(y_{j}),u_{j}^{d_{j}})_n=1$. We include the details of this third equality. Suppose that $\eta_{n}(y_{j})=v\varpi^{a}F^{\times n}$, where $v\in\Ocal^{\times}$ and $a\in\Z$. Then $\eta_{n}(d_{j}y_{j}) = v^{d_{j}}\varpi^{d_{j}a}F^{\times n}$. By \ref{Hyp_eta_int} we know that $n$ divides $d_{j}a$. So $(\eta_{n}(d_{j}y_{j}),u_{j})=(v^{d_{j}}\varpi^{d_{j}a},u_{j})=(\varpi^{d_{j}a},u_{j})(v,u_{j}^{^{d_{j}}})=1$, where the last equality follows because $n$ divides $d_{j}a$ and $u_{j}^{d_{j}}\in R$.
\end{proof}

\subsection{Some identities}

For the rest of this section, let us prove the equalities in Theorem \ref{Thm:Main_Thm}. Since we already know that $\Ocal^\times/ R$ is a free $\Z/ p^{\ell}\Z$ module of rank $[F:\Q_p]$, it suffices to show
\begin{Prop}\label{P:three_identities_mod_p^ell}
The following three isomorphisms hold:
    \begin{align*}
        Y_{Q,p^\ell}/ p^\ell Y&\cong (Y_{Q, n}/ nY)\otimes (\Z/ p^\ell\Z);\\
        Y/ Y_{Q,p^\ell}&\cong (Y/ Y_{Q, n})\otimes (\Z/ p^\ell\Z);\\
        Y_{Q,p^\ell}/ J_{p^\ell}&\cong (Y_{Q, n}/ J_n)\otimes (\Z/ p^\ell\Z),
    \end{align*}
    where the tensor product is over $\Z$.
\end{Prop}

\begin{proof}
    Recall $n=p^{\ell}m$ with $\gcd(p^{\ell}, m)=1$. We will construct the three isomorphisms:
    \[
    \begin{array}{rlrl}
     \Yt/nY\hspace{-.7em}&\cong Y_{Q,p^{\ell}}/p^{\ell}Y\oplus Y_{Q,m}/mY,& y+nY\hspace{-.7em}&\mapsto (y+p^{\ell}Y, y+mY);\\
     Y/\Yt\hspace{-.7em} &\cong Y/Y_{Q,p^{\ell}}\oplus Y/Y_{Q,m},& y+Y_{Q, n}\hspace{-.7em}&\mapsto (y+Y_{Q, p^{\ell}}, y+Y_{Q,m});\\
     \Yt/J_{n}\hspace{-.7em}&\cong Y_{Q,p^{\ell}}/J_{p^{\ell}}\oplus Y_{Q,m}/J_{m},& y+J_n\hspace{-.7em}&\mapsto (y+J_{p^\ell}, y+J_m),
    \end{array}
    \]
    where one can readily verify that the maps are all well-defined. Once these three isomorphisms are obtained, the proposition follows by tensoring with $\Z/p^{\ell}\Z$. Indeed, since $\gcd(p^{\ell}, m)=1$, there exist $a, b\in\Z$ such that $am+bp^\ell=1$. Hence for each $y\in Y$, we can write $y=amy+bp^{\ell}y$. Then after tensoring with $\Z/p^{\ell}\Z$, $bp^{\ell}y$ gets killed, and then since $amy\in mY$, the second factors on the right-hand sides all vanish.

    To prove the three isomorphism, consider the isomorphism
    \begin{equation*}
        f:Y/nY\xrightarrow{\sim} Y/p^{\ell}Y\oplus Y/mY,\quad y+nY\mapsto (y+p^{\ell}Y,y+mY),
    \end{equation*}
    given by the Chinese remainder theorem. Note that the inverse is explicitly given by
    \[
    (y_{1}+p^{\ell}Y,y_{2}+mY)\mapsto amy_{1}+bp^{\ell}y_{2}+nY,
    \]
    where $a$ and $b$ are as above.
    
    It is simple to check that $f(\Yt/nY) = Y_{Q,p^{\ell}}/p^{\ell}Y\oplus Y_{Q,m}/mY$, which proves the first isomorphism.

    For the second isomorphism, we can combine the isomorphisms $Y/nY\cong Y/p^{\ell}Y\oplus Y/mY$ and $\Yt/nY \cong Y_{Q,p^{\ell}}/p^{\ell}Y\oplus Y_{Q,m}/mY$.

    To show the third isomorphism, it suffices to show
    \begin{equation}\label{E:to_show_3rd_isom}
    f(J_{n}/nY) = J_{p^{\ell}}/p^{\ell}Y\oplus J_{m}/mY.
    \end{equation}
    Indeed, this identity gives us the isomorphism $J_{n}/nY\cong J_{p^{\ell}}/p^{\ell}Y\oplus J_{m}/mY$, and then the third isomorphism follows combining this isomorphism with the first one.
    
    Now, let us show \eqref{E:to_show_3rd_isom}. Recall that for $k\in\Z$ we have $J_{k} = Y_{Q,k}^{sc}+kY$, and $Y_{Q,k}^{sc}$ is generated by elements of the form $k_{\alpha}\alpha^{\vee}$, where $\alpha\in \Phi$ and $k_{\alpha}=\frac{k}{\gcd(k,Q(\alpha^{\vee}))}$. Note that because $\gcd(p^{\ell},m)=1$ we have that $n_{\alpha} = (p^{\ell})_{\alpha}m_{\alpha}$. From this it follows that $f(J_{n}/nY) \subset J_{p^{\ell}}/p^{\ell}Y\oplus J_{m}/mY$. For the reverse inclusion consider $m_{\alpha}\alpha^{\vee}\in Y_{Q,m}^{sc}$. Then we have 
    \begin{equation*}
        f^{-1}(0,m_{\alpha}\alpha^{\vee}+mY) = bp^{\ell}m_{\alpha}\alpha^{\vee}+nY=b\gcd(p^{\ell},Q(\alpha^{\vee}))n_{\alpha}\alpha^{\vee}+nY.
    \end{equation*}
    Thus $f^{-1}(J_{m}/mY)\subset J_{n}/nY$. Similarly, $f^{-1}(J_{p^{\ell}}/p^{\ell}Y)\subset J_{n}/nY$. The proof is complete.

\end{proof}


\section{Pseudo-spherical representations of $\Tt_0$}\label{S:Torus_Reps_I}


In this section, we give our proof of Theorem \ref{Thm:Main_Thm}.


\subsection{A general construction}


Let us recall the following general fact, which essentially says that the representations of covers of abelian groups are determined by their central characters.

\begin{Lem}\label{L:rep_of_2-step_unipotent_group}
Let $H$ be a compact abelian subgroup, and $\Ht$ an $n$-fold central extension of $H$, so that we have
\[
1\longrightarrow\mu_n\longrightarrow\Ht\longrightarrow H\longrightarrow 1.
\]
Let $\Irr_{\epsilon}\Ht$ be the set of (equivalence classes of) finite-dimensional irreducible $\epsilon$-genuine representations of $\Ht$ and $\Irr_{\epsilon}Z(\Ht)$ the set of $\epsilon$-genuine characters of the center $Z(\Ht)$. Then there is a bijection
\[
\Irr_{\epsilon}Z(\Ht)\xrightarrow{\;\sim\;}\Irr_{\epsilon}\Ht,\quad \chi\mapsto \pi(\chi),
\]
where $\pi(\chi)$ is the unique constituent of the isotypic representation $\Ind_{Z(\Ht)}^{\Ht}\chi$. Note that the central character of $\pi(\chi)$ is $\chi$.

Moreover, we have
\[
\Ind_{Z(\Ht)}^{\Ht}\chi=d\,\pi(\chi),
\]
where
\[
d=\dim\pi(\chi)=\left|\Ht/ Z(\Ht)\right|^{\frac{1}{2}}.
\]
\end{Lem}
\begin{proof}
See \cite[Proposition 2.2, p.706]{ABPTV}.
\end{proof}

Also if $W$ acts on $\Ht$ through automorphisms, so that $W$ also acts on $Z(\Ht)$, then for each $w\in W$ we have
\[
^{w}\left(\Ind_{Z(\Ht)}^{\Ht}\chi\right)\cong\Ind_{Z(\Ht)}^{\Ht}({^{w}\chi}).
\]
Hence the bijection of the lemma restricts to the bijection
\begin{equation}\label{E:bijection_Weyl_invariant_rep}
(\Irr_{\epsilon}Z(\Ht))^{W}\xrightarrow{\;\sim\;}(\Irr_{\epsilon}\Ht)^{W},
\end{equation}
where $(\Irr_{\epsilon}Z(\Ht))^{W}$ is the set of all $W$-invariant genuine characters of $Z(\Ht)$ and $(\Irr_{\epsilon}\Ht)^{W}$ is the set of (equivalence classes of) $W$-invariant genuine irreducible representations of $\Ht$.


\subsection{Decomposition of $\Tt_0/T_R$}\label{SS:Factor_T0}


By Lemma \ref{L:rep_of_2-step_unipotent_group}, the (finite-dimensional) irreducible $\epsilon$-genuine representation of $\Tt_0/ T_R$ are parameterized by $\epsilon$-genuine characters of $Z(\Tt_0/ T_R)$. Accordingly, we need a more concrete description of $Z(\Tt_0/ T_R)$. The first step is to decompose the group $\Tt_0/ T_R$ by using the decomposition of $\Ocal^\times/R$ of Proposition \ref{Prop:Hilb_Decomp}.

Recall that we have the identification
\[
Y\otimes F^\times\xrightarrow{\sim} T,\quad y\otimes t\mapsto y(t),
\]
where the tensor product is over $\Z$, and we have set
\[
T_0=Y\otimes\Ocal^\times\qand T_R=Y\otimes R.
\]
For each $y\in Y$ and $t\in F^\times$, we write
\[
\yt(t)=\s(y(t)).
\]
By \eqref{Eq:TSectionMult}, $\yt(u)$ for $u\in\Ocal^\times/ R$ is well-defined modulo $T_R$. 

Proposition \ref{Prop:Hilb_Decomp} provides a decomposition
\[
\Ocal^\times/ R=\la u_1\ra\oplus\cdots\oplus\la u_k\ra\;\oplus\;\la e_1, f_1\ra\oplus\cdots\oplus\la e_m, f_m\ra,
\]
where the rank 1 module $\la u_i\ra$ can appear only when $p^\ell=2$. For each $u_i\in\{u_1,\dots, u_k\}$, we define
\[
\Mt_i:=\la \epsilon\yt(u_i)\st y\in Y, \epsilon\in \mu_n\ra\subset \Tt_0/ T_R,
\]
and for each pair $\{e_i, f_i\}$, we define
\[
\Nt_i:=\la \epsilon\yt(e_i),\; \epsilon\yt(f_i)\st y\in Y, \epsilon\in\mu_n\ra\subset \Tt_0/ T_R,
\]
where we are identifying $T_R$ with $\s(T_R)\subset\Tt_0$.

Recall from \eqref{Comm_Id} the commutator identity in $\Tt$
\[
[\yt(u),\,\zt(v)]=(u, v)_n^{B_Q(y, z)}\in\mu_n
\]
for all $y, z\in Y$ and $u, v\in F^\times$. Since the orthogonal sum decomposition of $\Ocal^\times/ R$ is with respect to the Hilbert symbol $(-,-)_n$, we know that all the groups $\Mt_i$ and $\Nt_i$ commute with each other. Then we set
\[
\Mt:=\Mt_1\cdots\Mt_k\qand \Nt:=\Nt_1\cdots\Nt_{m}.
\]
Note that $\Mt$ is the group generated by all the elements of the form $\yt(u_i)$ viewed in $\Tt_0/ T_R$ and $\Nt$ is the one generated by all the elements of the form $\yt(e_i)$ and $\yt(f_i)$. Hence we have
\begin{equation}\label{E:exact_sequence_Tt}
1\longrightarrow T_{R}\stackrel{\mathbf{s}_{R}}{\longrightarrow}\Tt_{0}\longrightarrow \Mt\cdot \widetilde{N}\longrightarrow 1,
\end{equation}
so
\[
\Tt_0/T_R=\Mt\cdot\Nt=(\Mt_1\cdots\Mt_k)\cdot(\Nt_1\cdots\Nt_{m}).
\]


\subsection{On the center $Z(\Tt_0/ T_R)$}


With the above decomposition of $\Tt_0/ T_R$, we will give a more concrete description of $Z(\Tt_0/ T_R)$.

Let us first mention
\begin{Prop}\label{P:center_of_M_N}
\[
Z(\Mt\cdot\Nt)=Z(\Tt_0/ T_R)\cong Z(\Tt_0)/ T_R.
\]
\end{Prop}
\begin{proof}
The first identity is immediate from \eqref{E:exact_sequence_Tt}. To show the second equality, consider the natural map
\[
Z(\Tt_0)\longrightarrow Z(\Tt_0/ T_R).
\]
We can see this map is a surjection. Indeed, if $z\, T_R\in Z(\Tt_0/ T_R)$, where $z\in \Tt_0$, then for all $t\in\Tt_0$, we have
\[
[\,z,\, t\,]\in T_R.
\]
At the same time, $[\,z,\, t\,]\in\mu_n$, and $\mu_n\cap T_R$ is trivial. Hence $[\,z, t\,]=1$, showing that $z\in Z(\Tt_0)$. Now, the kernel of this map is $T_R\cap Z(\Tt_0)$. But $T_R\subset Z(\Tt_0)$ by Proposition \ref{Prop:TRSplits}, so that we have $Z(\Tt_0)/ T_R=Z(\Tt_0/ T_R)$. 
\end{proof}


Next, let us describe the centers $Z(\Mt_i)$ and $Z(\Nt_i)$ concretely as follows.
\begin{Lem}\label{L:isomorphism_of_Z(M)}
We have a group isomorphism
\[
Y_{Q, p^{\ell}}/ p^{\ell}Y\cong (Y_{Q, n}/ nY)\otimes \la u_i\ra\longrightarrow \pr(Z(\Mt_i))
\]
given by
\[
y\mapsto \pr(\yt(u_i)),\quad (y\in Y_{Q, p^{\ell}}/ p^{\ell}Y),
\]
where $\la u_i\ra$ is the cyclic group generated by $u_i$ in $\Ocal^\times/ R$, so $\la u_i\ra\cong \Z/ p^\ell\Z$.
(Here, we recall that this case happens only when $p^{\ell}=2$.)
\end{Lem}
\begin{proof}
In this proof, we omit the subscript $i$ for $\Mt_i$ and $u_i$, so we write $\Mt=\Mt_i$ and $u=u_i$. And we, of course, assume $p^\ell=2$.

Let $y\in Y_{Q, p^{\ell}}/ p^{\ell}Y$. Because of the first isomorphism of the lemma, which is  given by Proposition \ref{P:three_identities_mod_p^ell}, we may assume that $y$ is represented by an element in $Y_{Q, n}$. Then we already know that $\yt(u)\in Z(\Tt)$, and hence $\yt(u)\in Z(\Mt)$. So the map is well-defined. That this is a homomorphism is immediate.

To show this map is surjective, let $\yt(u)\in Z(\Mt)$. Then, for each $\zt(u)\in \Mt$, we must have
\[
(u, u)_n^{B_{Q}(y, z)}=1.
\]
Since $(u, u)_n=-1$, we must have $B_{Q}(y, z)\in 2\Z$ for all $z\in Y$. Thus $y\in Y_{Q,2}=Y_{Q, p^\ell}$. Hence the map is surjective. 

Now assume $y\in Y_{Q,p^{\ell}}$ is such that $\yt(u)=1$, so that $y(u)\in T_R$. Let $y_1,\dots,y_N\in Y$ be a $\Z$-basis of $Y$, and write $y=a_1y_1+\cdots+a_Ny_N$ for $a_i\in \Z$. Since $y_1,\dots,y_N$ are a $\Z$-basis, $y(u)\in T_R$ if and only if $u^{a_i}\in R$ for all $i=1,\dots, N$. So we must have $(u^{a_i}, u)_n=1$, namely $(u, u)_n^{a_i}=1$. But $(u, u)_n=-1$, which implies $a_i\in 2\Z=p^{\ell}\Z$. The lemma follows.
\end{proof}

\begin{Lem}\label{L:isomorphism_of_Z(N)}
We have a group isomorphism
\[
Y_{Q, p^{\ell}}/ p^{\ell}Y\times Y_{Q, p^{\ell}}/ p^{\ell}Y\cong (Y_{Q, n}/ nY)\otimes \la e_i, f_i\ra\longrightarrow \pr(Z(\Nt_i))
\]
given by
\[
(y_1, y_2)\mapsto \pr(\yt_1(e_i)\,\yt_2(f_i)),
\]
where $\la e_i, f_i\ra$ is the group generated by $e_i$ and $f_i$ in $\Ocal^\times/ R$, so $\la e_i, f_i\ra\cong \Z/ p^\ell\Z\times \Z/ p^\ell\Z$.
\end{Lem}
\begin{proof}
In this proof, we omit the subscript $i$ for $N_i, e_i, f_i$, and simply write $N, e, f$.

Because of the first isomorphism of the lemma, which is  given by Proposition \ref{P:three_identities_mod_p^ell}, we may assume that $y_1$ and $y_2$ are represented by elements in $Y_{Q, n}$. Then we already know that $\yt_1(e), \yt_2(f)\in Z(\Tt)$, and hence $\yt_1(e), \yt_2(f)\in Z(\Nt)$. So the map is well-defined. That this is a homomorphism is immediate.

To show this map is surjective, note that each element in $Z(\Nt)$ is of the form $\epsilon\yt_1(e)\yt_2(f)$ for some $y_1, y_2\in Y$ and $\epsilon\in\mu_n$. We then need to show $y_1, y_2\in Y_{Q, p^\ell}$. Let $\zt(v)\in\Nt$ be arbitrary, where $v=e$ or $f$. Then
\begin{equation}\label{E:commutator_relation_in_proof}
[\yt_1(e)\yt_2(f), \zt(v)]=(e, v)_n^{B_Q(y_1, z)}(f, v)_n^{B_Q(y_2, z)}=1.
\end{equation}

Now, set $v=e$. Then $(e, e)_n=1$ and $(f, v)=\zeta^{-1}$, where $\zeta$ is a primitive $p^\ell$-th root of unity by Proposition \ref{Prop:Hilb_Decomp}. Hence the above equation \eqref{E:commutator_relation_in_proof} gives $B_Q(y_2,z)\in p^\ell\Z$. Since $z\in Y$ is arbitrary, we have $y_2\in Y_{Q, p^\ell}$. Then the equation \eqref{E:commutator_relation_in_proof} gives $(e, v)_n^{B_Q(y_1, z)}=1$ for all $v$ and $z$. By setting $v=f$, so that $(e,f)_n=\zeta$, we know that $B_Q(y_1, z)\in p^\ell\Z$, namely $y_1\in Y_{Q, p^\ell}$.

To show that the map is injective, assume that $z_{1},z_{2}\in Y_{Q,p^{\ell}}$ are such that $\zt_{1}(e)\zt_{2}(f)=1$, so that $z_{1}(e)z_{2}(f)\in T_R$. Let $y_1,\dots,y_N\in Y$ be a $\Z$-basis of $Y$, and write $z_{1}=a_1y_1+\cdots+a_Ny_N$ and $z_{2}=b_1y_1+\cdots+b_Ny_N$ for $a_i,b_{i}\in \Z$. Since $y_1,\dots,y_N$ are a $\Z$-basis, $z_{1}(e)z_{2}(f)\in T_R$ if and only if $e^{a_{i}}f^{b_{i}}\in T_R$ for all $i=1,\dots, N$. So we must have $(e^{a_i}f^{b_{i}}, e)_n=1$, which implies $(f, e)_n^{b_i}=1$. But $(e, f)_n=\zeta$. So we have $b_i\in p^{\ell}\Z$. We also have $(e^{a_i}f^{b_{i}}, f)_n=1$, which, given $b_i\in p^{\ell}\Z$, implies $(e^{a_i}, f)_n=1$. Thus $a_i\in p^{\ell}\Z$ and the injectivity follows.
\end{proof}

These two lemmas give us

\begin{Prop}\label{P:description_of_Z(T_0)_T_R}
    We have a group isomorphism
    \[
    (Y_{Q, n}/ nY)\otimes (\Ocal^\times/ R)\longrightarrow \pr(Z(\Tt_0)/ T_R)=\pr(Z(\Tt_0/ T_R))
    \]
    given by
    \[
    y\otimes u\mapsto \pr(\yt(u)).
    \]
\end{Prop}
\begin{proof}
Note that
\[
Z(\Tt_0)/ T_R=Z(\Tt_0/ T_R)=Z(\Mt\cdot \Nt)=Z(\Mt_1)\cdots Z(\Mt_k)\cdot Z(\Nt_1)\cdots Z(\Nt_m).
\]
Hence the result follows from the previous two lemmas together with the decomposition
\[
\Ocal^\times/ R=\la u_1\ra\oplus\cdots\oplus\la u_k\ra\oplus \la e_1,f_1\ra\oplus\cdots\oplus \la e_m, f_m\ra 
\]
of Proposition \ref{Prop:Hilb_Decomp}.
\end{proof}

Note that $Z(\Tt)\cap\Tt_0\subset Z(\Tt_0)$ but in general $Z(\Tt)\cap\Tt_0\neq Z(\Tt_0)$. But the following corollary says that the equality holds modulo $T_R$.
\begin{Cor}\label{C:centers_are_equal_mod_T_R}
    The inclusion $Z(\Tt)\cap\Tt_0\subset Z(\Tt_0)$ induces the isomorphism
    \[
    (Z(\Tt)\cap \Tt_0)/(Z(\Tt)\cap T_R)\cong Z(\Tt_0)/T_R.
    \]
    In particular, $Z(\Tt_{0})$ is generated by its subgroups $Z(\Tt)\cap T_{0}$ and $T_{R}$. 
\end{Cor}
\begin{proof}
    The inclusion $Z(\Tt)\cap\Tt_0\subset Z(\Tt_0)$ induces the inclusion
    \[
    (Z(\Tt)\cap \Tt_0)/(Z(\Tt)\cap T_R)\subset Z(\Tt_0)/T_R,
    \]
    noting $Z(\Tt)\cap \Tt_0\cap T_R=Z(\Tt)\cap T_R$. We will show this is an equality. The group $Z(\Tt)\cap \Tt_0$ is generated by elements of the form $\ybb(u)$ for $y\in Y_{Q, n}$ and $u\in\Ocal^\times$, by \cite[Proposition 4.1]{W09}. Since $Z(\Tt_0)/T_R$ is generated by elements of the form $\ybb(u)T_R$ for $y\in Y_{Q, n}$ and $u\in\Ocal^\times$ by the above proposition, we know that the above inclusion is an equality.
\end{proof}


\subsection{Pseudo-spherical representations of $\Tt_0$}


Once we know the size of the center $Z(\Tt_0/ T_R)$, one can readily compute the dimension and the number of pseudo-spherical representations of $\Tt_0$ as follows, proving the first part of Theorem \ref{Thm:Main_Thm}.
\begin{Prop}
The number of the pseudo-spherical representations of $\Tt_{0}$ is
\[
\Big|Y_{Q,p^{\ell}}/ p^{\ell}Y\Big|^{[F:\Q_p]}=\Big|(Y_{Q, n}/ nY)\otimes (\Ocal^\times/ R)\Big|.
\]
If $\tau$ is a pseudo-spherical representation of $\Tt_0$, then
\[
\dim\tau=\Big|Y/ Y_{Q, p^{\ell}}\Big|^{\frac{[F:\Q_p]}{2}}=\Big|(Y/ Y_{Q, n})\otimes(\Ocal^\times/ R)\Big|^{1/2}.
\]
\end{Prop}
\begin{proof}
The number of pseudo-spherical representations of $\Tt_0$ is the number of $\epsilon$-genuine characters of the center $Z(\Tt_0/T_R)$ and hence the cardinality of the group $\pr(Z(\Tt_0/ T_R))$ by Lemma \ref{L:rep_of_2-step_unipotent_group}. Hence the proposition follows from Proposition \ref{P:description_of_Z(T_0)_T_R}.

Also by Lemma \ref{L:rep_of_2-step_unipotent_group},
\[
\dim \tau
=\Big|(\Tt_0/ T_R)/ Z(\Tt_0/ T_R)\Big|^{1/2}
\]
We know $Y\otimes(\Ocal^\times/ R)\cong \pr(\Tt_0/ T_R)$, and also
$|Y\otimes(\Ocal^\times/ R)|=|Y/ p^{\ell}Y|^{[F:\Q_p]}$ because $\Ocal^\times/ R\cong (\Z/ p^\ell\Z)^{[F:\Q_p]}$ and $Y\otimes \Z/ p^\ell\Z=Y/ p^{\ell} Y$. Also we already know $|\pr(Z(\Tt_0/ T_R)|=|Y_{Q, p^{\ell}}/ p^{\ell}Y|^{[F:\Q_p]}$. Hence
\[
(\dim\tau)^2=\frac{\big|Y/ p^{\ell}Y\big|^{[F:\Q_p]}}{\big|Y_{Q, p^{\ell}}/ p^{\ell}Y\big|^{[F:\Q_p]}}=\Big|Y/ Y_{Q, p^{\ell}}\Big|^{[F:\Q_p]}.
\]
The proposition follows.
\end{proof}


\subsection{Distinguished pseudo-spherical representations of $\Tt_0$}


Let us count the number of distinguished pseudo-spherical representations of $\Tt_0$. Recall that a pseudo-spherical representation $\tau$ of $\Tt_0$ is distinguished if its central character is distinguished in the sense that its pullback via $\Tt_{Q, n}(\Ocal)\to Z(\Tt_0)$ is trivial on both $\s_\eta(T_{Q, n}^{sc}(\Ocal))$ and $\gt(T_0)$.

Now, the isomorphism of Proposition \ref{P:description_of_Z(T_0)_T_R} does not lift to $Z(\Tt_0)/T_R$. By restricting to the domain $(Y_{Q, n}/nY)\otimes(\Ocal^\times/R)$ to a smaller group, however, it does lift, assuming \ref{Hyp_eta_ext} and \ref{Hyp_eta_int} as follows.

\begin{Prop}\label{P:splitting_S_J_center}
Assume \ref{Hyp_eta_ext} and \ref{Hyp_eta_int}. We have an injective homomorphism
\[
\varphi:(J_n/ nY)\otimes (\Ocal^\times/R)\longrightarrow Z(\Tt_0)/ T_R, \qquad y\otimes u\mapsto \s_J(y(u)),
\]
where recall $\s_J:T_{Q, n}^J\to\Tt_{Q, n}$ is as given in \eqref{E:section_S_J}.

Further, the image of $\varphi$ is the group generated by images of $\s_\eta(T_{Q, n}^{sc}(\Ocal))$ and $\gt(T_0)$ under the map $\Tt_{Q,n}(\Ocal)\to Z(\Tt_0)/T_R$.
\end{Prop}
\begin{proof}
Since we are assuming \ref{Hyp_eta_ext}, the splitting $\s_J$ exists. Since we are also assuming \ref{Hyp_eta_int} the splitting is compatible with $\s_R$, so we know that $\s_J(y(u))$ is independent of the choice of $u$. The homomorphism is injective because it is a lift of an injection.

The assertion for the image of $\varphi$ is clear from the definitions of the maps involved.
\end{proof}

Hence we have proved the second part of Theorem \ref{Thm:Main_Thm}.
\begin{Prop}
Assume \ref{Hyp_eta_ext} and \ref{Hyp_eta_int}. Then the number of distinguished pseudo-spherical representations of $\Tt_0$ is
\[
\Big|Y_{Q,p^{\ell}}/ J_{p^{\ell}}\Big|^{[F:\Q_p]}=\Big|(\Yt/J_{n})\otimes (\Ocal^{\times}/R)\Big|.
\]
\end{Prop}
\begin{proof}
The distinguished pseudo-spherical representations are in one-to-one correspondence with the genuine characters of $Z(\Tt_0)/ T_R$ which are trivial on the image of $\varphi$, where $\varphi$ is from Proposition \ref{P:splitting_S_J_center}. Hence the number is equal to
\[
\frac{|(Y_{Q,n}/ nY)\otimes\Ocal^\times/ R|}{|(J_n/ nY)\otimes\Ocal^\times/ R|}.
\]
We know
\[
\Big|(Y_{Q,n}/ nY)\otimes\Ocal^\times/ R\Big|=\Big|Y_{Q, p^\ell}/ p^\ell Y\Big|^{[F:\Q_p]}.
\]
Also by the Chinese remainder theorem, we have
\[
J_n/nY\cong J_{p^{\ell}}/p^{\ell}Y\oplus J_{m}/m Y.
\]
Hence
\[
\Big|(J_n/ nY)\otimes\Ocal^\times/ R\Big|=\Big|J_{p^{\ell}}/p^{\ell}Y \Big|^{[F:\Q_p]}.
\]
Hence
\begin{align*}
\frac{|(Y_{Q,n}/ nY)\otimes\Ocal^\times/ R|}{|(J_n/ nY)\otimes\Ocal^\times/ R|}&=\left(\frac{|Y_{Q, p^\ell}/ p^\ell Y|}{|J_{p^{\ell}}/p^{\ell}Y|}\right)^{[F:\Q_p]}=\Big|Y_{Q, p^\ell}/J_{p^\ell}\Big|^{[F:\Q_p]}\\
&=\Big|(\Yt/J_{n})\otimes (\Ocal^{\times}/R)\Big|.
\end{align*}
The proposition follows.
\end{proof}

The following is an important property of distinguished pseudo-spherical representations, which finishes our proof of Theorem \ref{Thm:Main_Thm}.
\begin{Prop}
Assume \ref{Hyp_eta_ext} and \ref{Hyp_eta_int}, and let $\tau$ be a distinguished pseudo-spherical representation of $\Tt_0$.  Then $\tau$ is $W$-invariant.
\end{Prop}
\begin{proof}
By \cite[Theorem 6.8 and (6.7)]{GG18}, we know that
for each $\alpha\in\Phi$ and $y\in Y_{Q, n}$ we have
\[
-\la\alpha, y\ra\alpha^\vee\in Y_{Q,n}^{sc},
\]
and further for all $t\in\Ocal^\times$
\[
\wt_\alpha(1)\cdot\yt(t)\cdot\wt_\alpha(-1)=\yt(t)\cdot\s_{\eta}(-\la\alpha, y\ra\alpha^\vee(t)).
\]

Hence if $\chi_\tau$ is the character of $Z(\Tt_0)/ T_R$ corresponding to $\tau$, so that it is trivial on the elements of the form $\s_J(y(u))$ for $y\in J_n=Y_{Q, n}^{sc}+nY$, then it is trivial on the elements of the form $\s_{\eta}(-\la\alpha, y\ra\alpha^\vee(t))$ because $\s_J$ extends $\s_\eta$ (Subsection \ref{SS:Dist_Char}). Thus $\chi_\tau$, and hence $\tau$, is $W$-invariant.
\end{proof}


\section{Restriction from $\Tt$ to $\Tt_0$}\label{S:Torus_Reps_II}


In this section we consider branching from $\Tt$ to $\Tt_0$. In particular, we will show that restriction of an arbitrary (genuine) representation of $\Tt$ to $\Tt_0$ is always multiplicity free, and  restriction of a pseudo-spherical $\Tt$-representation always contains a unique pseudo-spherical $\Tt_{0}$-constituent. Along the way, we will show that every irreducible genuine representation of $\Tt_{0}$ is a type in the sense of Bushnell-Kutzko.

Let us recall the following notation. For a lattice $L\subset Y$ we have set
\[
T^L:=\text{image of $L\otimes F^\times\to Y\otimes F^\times$}.
\]

\subsection{Branching from $\Tt$ to $\Tt_0$}
In this subsection, we consider restriction of an irreducible genuine representation of $\Tt$ to $\Tt_0$; in particular, we will show that the restriction is always multiplicity free.

Let us start with
\begin{Lem}\label{Lem:TorusSubIsos}
    Let $r_{0}\in \Ocal^{\times}$ be defined as in Lemma \ref{Lem:IntTriv}. Let $L$ be a lattice such that $nY\subset L\subset Y$. Let us also write $T_{\Ocal^{\times n}}=T^{nY}\cap T_{0}$. 
\begin{enumerate}[label=(\arabic*), leftmargin=*]
        \item\label{Lem:TorusSubIsos:Arb_Lattice1}
        The map $f_{L}:L\rightarrow (T^{L} \cap T_{R})/T_{\Ocal^{\times n}}$ defined by $\ell\mapsto \ell(r_{0})T_{\Ocal^{\times n}}$ induces an isomorphism
        \begin{equation*}
             L/nY\cong (T^{L}\cap T_{R})/T_{\Ocal^{\times n}}. 
        \end{equation*}
        \item\label{Lem:TorusSubIsos:Arb_Lattice2}
        The map $f_{L}^{\prime}:Y\rightarrow T_{R}/(T^{L}\cap T_{R})$ defined by $y\mapsto y(r_{0})T^{L}\cap T_{R}$ induces an isomorphism
        \begin{equation*}
             Y/L\cong T_{R}/T^{L}\cap T_{R}. 
        \end{equation*}
\end{enumerate}
\end{Lem}

\begin{proof}
    We choose a basis for $Y$ to see that the kernel of $f_{L}$ is exactly $nY$.

    To prove that $f_{L}$ is surjective we choose compatible bases for $Y$ and $L$ as follows. Let $\{y_{j}\}$ be a basis for $Y$ such that $\{d_{i}y_{i}\}$ is a basis for $L$ for some $d_{i}\in\Z$. Note that $d_{i}$ divides $n$, because $nY\subset L\subset Y$. Suppose that $k_{j}\in\Z$ and $u_{j}\in\Ocal^{\times}$ are such that \[ \prod_{j}y_{j}(u_{j}^{d_{j}})T_{\Ocal^{\times n}}=\prod_{j}y_{j}(r_{0}^{k_{j}})T_{\Ocal^{\times n}} \in (T^{L}\cap T_{R})/T_{\Ocal^{\times n}}.\]
    Thus for all $j$ we have $ u_{j}^{d_{j}}=r_{0}^{k_{j}} $ modulo $\Ocal^{\times n}$. Raising both sides of the equation to the $\frac{n}{d_{j}}$-th power shows $k_{j}\frac{n}{d_{j}}$ is divisible by $n$. So $d_{j}$ divides $k_{j}$. Thus $\sum_{j}k_{j}y_{j}\in L$ and this element maps onto $\prod_{j}y_{j}(r_{0}^{k_{j}})T_{\Ocal^{\times n}}$, proving that the map is surjective.

    Item \ref{Lem:TorusSubIsos:Arb_Lattice2} directly follows from Item \ref{Lem:TorusSubIsos:Arb_Lattice1}, since $Y/L=(Y/nY)/(L/nY)$.
\end{proof}

Next, we define a family of characters $\varphi_y$ of $T$. For $y\in Y$, define
\[
\varphi_y:T\longrightarrow\C^\times,\quad t\mapsto \epsilon([\widebar{t}, \widebar{y}(\varpi)]).
\]
Certainly $\varphi_y$ is trivial on $\pr(Z(\Tt))$. Also note that $\varphi_{y}|_{T_{R}}$ does not depend on the choice of uniformizer $\varpi$.

\begin{Lem}\label{Lem:TR/ZTDual}
    The map $Y\rightarrow \mathrm{Hom}(T_{R}/\left(\pr(Z(\Tt))\cap T_{R}\right),\,\mathbb{C}^{\times})$ defined by $y\mapsto \varphi_{y}|_{T_{R}}$ induces an isomorphism of abelian groups
    \[Y/\Yt\cong \mathrm{Hom}(T_{R}/\left(\pr(Z(\Tt))\cap T_{R}\right),\,\mathbb{C}^{\times}).\]
\end{Lem}

\begin{proof}
    One can directly check the map is a homomorphism. Equation \eqref{Comm_Id} implies that the kernel of the map $y\mapsto \varphi_{y}|_{T_{R}}$ is $\Yt$. Thus we have an injection $Y/\Yt\hookrightarrow \mathrm{Hom}(T_{R}/\left(\pr(Z(\Tt))\cap T_{R}\right),\,\mathbb{C}^{\times})$. By Lemma \ref{Lem:TorusSubIsos} \ref{Lem:TorusSubIsos:Arb_Lattice2}, the two groups have the same size and so the map is also surjective.
\end{proof}

\begin{Thm}\label{Thm:CmptBranch}
	Let $\sigma\in\Irr_{\epsilon}(\overline{T})$. Let $\sigma_{0}\in\Irr_{\epsilon}(\overline{T}_{0})$ be a fixed constituent of $\sigma|_{\overline{T}_{0}}$. Then as $\overline{T}_{0}$-representations
	\[\sigma|_{\overline{T}_{0}}\cong \bigoplus_{y\in Y/\Yt}\varphi_{y}\otimes \sigma_{0}.\]
	
	Moreover, $\varphi_{y}\otimes \sigma_{0}\cong \sigma_{0}$ if and only if $y\in \Yt$. Hence the restriction is multiplicity free.
\end{Thm}

\begin{proof}
Note that by Lemma \ref{Lem:TR/ZTDual}, $\varphi_{y}\otimes \sigma_{0}\cong \sigma_{0}$ if and only if $y\in \Yt$, because $\varphi_{y}$ changes the central character unless $y\in\Yt$, and the representation of $\Tt_0$ is determined by its central character.

Now for each $y\in Y$, the subspace $\sigma(\ybb(\varpi))\cdot \sigma_{0}\subset \sigma$ is $Z(\Tt)\overline{T}_{0}$-invariant and, as a $\overline{T}_{0}$-representation, $\sigma(\ybb(\varpi))\cdot \sigma_{0}\cong \varphi_{y}\otimes \sigma_{0}$. Hence we have
\[
\bigoplus_{y\in Y/\Yt}\varphi_{y}\otimes \sigma_{0}\subset \sigma|_{\Tt_0}.
\]
But the left-hand side is $\Tt$-invariant since $Y/ Y_{Q, n}\cong \Tt/ \Tt_0Z(\Tt)$ via the map $y\mapsto \ybb(\varpi)$, and hence equals the right-hand side by irreducibility of $\sigma$. 
\end{proof}

\begin{Rmk}
Theorem \ref{Thm:CmptBranch} shows that \cite[Proposition 6.1]{HW25} holds for general split tori, without the decomposing Levis hypothesis.
\end{Rmk}

\subsection{Type theory}

The above theorem also implies that an irreducible genuine representation of $\Tt_0$ is always a type in the sense of Bushnell-Kutzko. To see it, we need to consider how to construct a representation of $\Tt$ from one of $\Tt_0$ first by extending to $Z(\Tt)\Tt_0$ by adding a character of $Z(\Tt)$, and second by inducing to $\Tt$. This can be done as follows.

Let $\Yt(\varpi)\subset T$ be the subgroup of elements of the form $y(\varpi)$, where $y\in \Yt$. Note that $\Yt(\varpi)$ depends on the choice of $\varpi$, and the preimage $\overline{\Yt(\varpi)}\subset Z(\Tt)$.

\begin{Lem}\label{Lem:FactorZT0}
	We have a short exact sequence of groups
	\begin{equation*}
		\mu_{n}\hookrightarrow \overline{\Yt(\varpi)}\times \overline{T}_{0}\twoheadrightarrow Z(\overline{T})\overline{T}_{0},
	\end{equation*}
	where the injection embeds $\mu_{n}$ anti-diagonally and the surjection is given by multiplication.
\end{Lem}

\begin{proof}
	
	This follows directly from \cite[Proposition 4.1]{W09}.
\end{proof}

Lemma \ref{Lem:FactorZT0} makes clear how to extend a representation of $\overline{T}_{0}$ to $Z(\overline{T})\overline{T}_{0}$; all one needs is to construct a genuine character of $\overline{\Yt(\varpi)}\subset Z(\Tt)$.

\begin{Prop}\label{Prop:Tt0Types}
	Let $\sigma\in \Irr_{\epsilon}(\Tt)$ with central character $\chi$. Let $\sigma_{0}\in\Irr_{\epsilon}(\Tt_{0})$. The restriction of the $\overline{T}$-representation $\sigma$ to $\overline{T}_{0}$ contains $\sigma_{0}$ if and only if $\sigma\cong ind_{Z(\overline{T})\overline{T}_{0}}^{\overline{T}}(\chi|_{\overline{\Yt(\varpi)}}\otimes \sigma_{0})$.	
\end{Prop}

\begin{proof}
	This follows from Theorem \ref{Thm:CmptBranch}, Lemma \ref{Lem:FactorZT0}, Frobenius reciprocity, and Mackey theory. 
\end{proof}

Proposition \ref{Prop:Tt0Types} and \cite[(5.2) Proposition]{BK98} give the following corollary.

\begin{Cor}\label{Cor:BK_Type}
    Let $\sigma\in \Irr_{\epsilon}(\Tt)$ with a $\Tt_{0}$-constituent $\sigma_{0}$. Then the pair $(\Tt_{0},\sigma_{0})$ is a type for the inertial class $[\Tt,\sigma]_{\Tt}$, in the sense of Bushnell-Kutzko \cite{BK98}.
\end{Cor}
This corollary generalizes the analogous type result in \cite{Wang24}, where tame covers of split tori are considered.

\subsection{Restriction of pseudo-spherical representations}

Now, we consider restriction of pseudo-spherical representations of $\Tt$ to $\Tt_0$. Recall we denote by $\Irr_{\epsilon}^{ps}(\Tt)$ the set of pseudo-spherical representations of $\Tt$.

Also let us agree on the following terminology. For each abelian group $A$ with $\mu_n\subset A\subset \Tt$, we say a character of $A$ is pseudo-spherical if it is genuine and trivial on $A\cap T_R$. Further, assume that $Z(\Tt)$ admits a distinguished character, so that we have the splitting $\s_J:T^J\rightarrow \Tt$. Then we say a character of $A$ is distinguished if it is genuine and trivial on $A\cap \s_J(T^J)$. We denote the set of pseudo-spherical characters of $A$ by $\Irr_{\epsilon}^{ps}(A)$ and the set of distinguished pseudo-spherical characters by $\Irr_{\epsilon}^{dps}(A)$.

\begin{Prop}\label{Prop:TRFixedConstit}
    Let $\sigma\in\Irr_{\epsilon}^{ps}(\overline{T})$ with central character $\chi$. 
    \begin{enumerate}[leftmargin=*]
        \item\label{Prop:It1:TRFixedConstit} The restriction $\sigma|_{\Tt_{0}}$ has a unique pseudo-spherical constituent $\sigma_0$. 
        \item\label{Prop:It2:TRFixedConstit} Suppose that $T\subset G$ is a maximal split torus with Weyl group $W$ and assume \ref{Hyp_eta_int}. If $\chi$ is $W$-invariant, then the constituent $\sigma_{0}$ is also $W$-invariant.
    \end{enumerate}
\end{Prop}
\begin{proof}
    By Theorem \ref{Thm:CmptBranch} there exists $\sigma_{0}^{\prime} \in\Irr_{\epsilon}(\Tt_{0})$ such that
    \begin{equation}\label{Lem:TRFixedConstitEq1}
        \sigma|_{\Tt_{0}}\cong \bigoplus_{y\in Y/\Yt} \varphi_{y}\otimes \sigma_{0}^{\prime}.
    \end{equation}
    Let $\chi_{0}^{\prime}$ be the central character of $\sigma_{0}^{\prime}$. Then $T_{R}$ acts on $\varphi_{y}\otimes \sigma_{0}^{\prime}$ by $\varphi_{y}\cdot \chi_{0}^{\prime}|_{T_{R}}$. By construction each $\varphi_{y}$ is trivial on $\pr(Z(\Tt))\cap T_{R}$, and $\chi_{0}^{\prime}|_{T_{R}}$ is also trivial on $Z(\Tt)\cap T_{R}$ since $\chi$ is trivial on $Z(\Tt)\cap T_{R}$. So, by Lemma \ref{Lem:TR/ZTDual}, there exists a unique $y_{0}\in Y/\Yt$ such that $T_{R}$ acts trivially on $\varphi_{y_{0}}\otimes \sigma_{0}^{\prime}$. Thus $\sigma_{0}:=\varphi_{y_{0}}\otimes \sigma_{0}^{\prime}$ is the unique pseudo-spherical constituent, proving Item \eqref{Prop:It1:TRFixedConstit}.

    Now we prove \eqref{Prop:It2:TRFixedConstit}. By assumption $\,^{w}\sigma\cong \sigma$. As we just proved, $\sigma_{0}\subset \sigma$ is the unique pseudo-spherical $\Tt_{0}$-constituent. Since we are assuming \ref{Hyp_eta_int}, Proposition \ref{Prop:TRSplits} implies $T_{R}$ is $W$-invariant. So $\,^{w}\sigma_{0}\subset\sigma$ is also a pseudo-spherical $\Tt_{0}$-constituent. Thus $\sigma_{0}\cong \,^{w}\sigma_{0}$ by the uniqueness of pseudo-spherical constituents.
\end{proof}

From Proposition \ref{Prop:TRFixedConstit} \eqref{Prop:It1:TRFixedConstit} we get a characterization of the pseudo-spherical $\Tt$-representations in terms of their central character.

\begin{Cor}\label{Cor:PS_Char}
    Let $\sigma\in\Irr_{\epsilon}(\overline{T})$ be a $\Tt$-representation. Then $\sigma$ is pseudo-spherical if and only if the central character of $\sigma$ is pseudo-spherical. In other words, we have the bijection
    \[
    \Irr_{\epsilon}^{ps}(\Tt)\xrightarrow{\sim}\Irr_{\epsilon}^{ps}(Z(\Tt)),\quad \tau\mapsto \chi_\tau,
    \]
    where $\chi_{\tau}$ is the central character of $\tau$. Further, assuming \ref{Hyp_eta_ext} and \ref{Hyp_eta_int} this bijection restricts the bijection
    \[
    \Irr_{\epsilon}^{dps}(\Tt)\xrightarrow{\sim}\Irr_{\epsilon}^{dps}(Z(\Tt)).
    \]
\end{Cor}
\begin{proof}
    The first bijection follows from the proof of Proposition \ref{Prop:TRFixedConstit} \eqref{Prop:It1:TRFixedConstit}. The second then follows from the definition of distinguished pseudo-spherical representation.
\end{proof}

This corollary also implies the following existence assertion.
\begin{Prop}
    Assume  \ref{Hyp_eta_ext} and \ref{Hyp_eta_int}. Then there exists a distinguished pseudo-spherical character of $Z(\Tt)$, and hence there exists a distinguished pseudo-spherical representation by the above corollary.
\end{Prop}
\begin{proof}
    Since $T_R$ splits to $\Tt$, we have $\Tt_R\cong T_R\times\mu_n$. Hence there exists a pseudo-spherical character $\chi$ on $\Tt_R$, namely the projection $\Tt_R\to\mu_n$ composed with $\epsilon$. By restricting, we obtain a pseudo-spherical character of $\Tt_R\cap Z(\Tt)$, namely a character of $\Tt_R\cap Z(\Tt)$ trivial on $\s_J(T_R\cap T^J)$ by the compatibility of the two splittings $\s_J$ and $\s_R$. 
    
    
    Now, via the inclusion
    \[
    (\Tt_R\cap Z(\Tt))/\s_J(T_R\cap T^J)\subset Z(\Tt)/\s_J(T^J),
    \]
    one can extend $\chi|_{Z(\Tt)\cap\Tt_R}$ to a genuine character of $Z(\Tt)$ which is trivial on $\s_J(T^J)$. By construction this extension is trivial on both $\s_J(T^J)$ and $\s_R(T_R)\cap Z(\Tt)$. The lemma is proved.
\end{proof}

\begin{Rmk}
    In the proof above, the compatibility of $\s_J$ and $\s_R$ plays a crucial role. Indeed, if they do not agree on the overlap, then any character of $Z(\Tt)$ trivial on both $\s_J(T^J)$ and $\s_R(T_R)$ cannot be genuine.
\end{Rmk}

Next, we consider the restriction problem of distinguished pseudo-spherical representations of $\Tt$ to $\Tt_0$. For this purpose, recall that the (genuine) representations of $\Tt_0$ are determined by their central characters (Lemma \ref{L:rep_of_2-step_unipotent_group}). Namely we have the bijection
\[
    \Irr_{\epsilon}^{ps}(\Tt_0)\xrightarrow{\sim}\Irr_{\epsilon}^{ps}(Z(\Tt_0)),\quad \tau\mapsto \chi_\tau,
\]
where $\chi_{\tau}$ is the central character of $\tau$. Also note that in general we have $Z(\Tt_0)\nsubset Z(\Tt)$ because $T_R\subset Z(\Tt_0)$, but in general $T_R\nsubset Z(\Tt)$ although we always have $Z(\Tt)\cap\Tt_0\subset Z(\Tt_0)$. But by Corollary \ref{C:centers_are_equal_mod_T_R} we know that the (distinguished) pseudo-spherical characters of $Z(\Tt)\cap\Tt_0$ and $Z(\Tt_0)$ coincide. Namely,

\begin{Prop}
    The map
    \[
    \Irr_{\epsilon}^{ps}(\Tt_{0})\cong\Irr_{\epsilon}^{ps}(Z(\Tt_0))\rightarrow\Irr_{\epsilon}^{ps}(Z(\Tt)\cap \Tt_{0})
    \]
    given by
    \[
    \tau\mapsto \chi_{\tau}\mapsto\chi_{\tau}|_{Z(\Tt)\cap \Tt_{0}}
    \]
is a bijection. Further, assume that there exists a distinguished pseudo-spherical character of $Z(\Tt)$, which happens if hypotheses \ref{Hyp_eta_ext} and \ref{Hyp_eta_int} are satisfied. Then this bijection restricts to the bijection
\begin{equation*}
    \Irr_{\epsilon}^{dps}(\Tt_{0})\cong \Irr_{\epsilon}^{dps}(Z(\Tt_0))\rightarrow \Irr_{\epsilon}^{dps}(Z(\Tt)\cap \Tt_{0}).
\end{equation*} 
\end{Prop}
\begin{proof}
    That the map $\Irr_{\epsilon}^{ps}(Z(\Tt_0))\rightarrow\Irr_{\epsilon}^{ps}(Z(\Tt)\cap \Tt_{0})$ is a bijection follows from the isomorphism
    \[
    (Z(\Tt)\cap \Tt_0)/(Z(\Tt)\cap T_R)\cong Z(\Tt_0)/T_R
    \]
    of Corollary \ref{C:centers_are_equal_mod_T_R}.

    Once we know the restriction map $\Irr_{\epsilon}^{ps}(Z(\Tt_0))\rightarrow\Irr_{\epsilon}^{ps}(Z(\Tt)\cap \Tt_{0})$ is a bijection, it is immediate that this bijection restricts to the bijection $\Irr_{\epsilon}^{dps}(Z(\Tt_0))\xrightarrow{\sim}\Irr_{\epsilon}^{dps}(Z(\Tt)\cap \Tt_{0})$.
\end{proof}

\begin{Cor}
    Let $\sigma\in\Irr_{\epsilon}^{dps}(\Tt)$. Then the restriction $\sigma|_{\Tt_0}$ contains a unique distinguished pseudo-spherical constituent.
\end{Cor}
\begin{proof}
    Since $\sigma$ is distinguished pseudo-spherical, its central character $\chi_\sigma$ is distinguished pseudo-spherical. Hence its restriction to $Z(\Tt)\cap\Tt_0$ is distinguished pseudo-spherical, which uniquely extends to a distinguished pseudo-spherical character of $Z(\Tt_0)$ by the above proposition. Hence the unique pseudo-spherical constituent $\sigma_0$ of Proposition \ref{Prop:TRFixedConstit} is distinguished.
\end{proof}


\section{Hecke algebras}\label{S:Hecke_Alg}


In this section we give an explicit description of the Hecke algebra of a genuine irreducible $\Tt_{0}$-representation.

Let $\tau\in \Irr_{\epsilon}(\Tt_{0})$ and let $\tau^{\vee}$ be its contragredient. Let $\Hcal=\Hcal(\Tt,\tau^{\vee})$ be the Hecke algebra attached to $\tau$. It consists of compactly supported functions $f:\Tt\rightarrow \mathrm{End}(\tau)$ such that for all $\gamma_{j}\in \Tt_{0}$ and $t\in \Tt$ we have
\begin{equation*}
    f(\gamma_{1}t\gamma_{2}) = \tau(\gamma_{1})f(t)\tau(\gamma_{2}).
\end{equation*}
Convolution defines a multiplication on $\Hcal$. Specifically, we fix a Haar measure on $\Tt$ such that the measure of $\Tt_{0}$ is $1$. Then given $f,g\in\Hcal$ we define the convolution of $f$ with $g$ to be 
\begin{equation*}
    f*g(t):=\int_{\Tt}f(u)g(u^{-1}t)du.
\end{equation*}

\begin{Lem}\label{Lem:Hecke_Supp}
    Let $y\in Y$. Let $\Hcal_{y}\subset \Hcal$ be the subspace of functions with support contained in the set $\ybb(\varpi)\Tt_{0}$, where $\ybb(\varpi)\in \Tt$ is any lift of $y(\varpi)\in T$. Then 
    \[
    \dim\Hcal_{y}=\begin{cases}1&\text{if $y\in \Yt$};\\0&\text{if $y\notin \Yt$}.\end{cases}
    \]
\end{Lem}

\begin{proof}
    From \cite[(4.1.1) Proposition]{Bushnell-Kutzko-Book} we know there is a canonical isomorphism of vector spaces between $\Hcal_{y}$ and \[\mathrm{Hom}_{\Tt_{0}}(\,^{y(\varpi)}\tau,\tau)\cong \mathrm{Hom}_{\Tt_{0}}(\varphi_{y}\otimes\tau,\tau).\] Since $\tau$ is irreducible, this space has dimension at most one by Schur's Lemma; by Lemma \ref{Lem:TR/ZTDual} the dimension is one if and only if $y\in \Yt$.
\end{proof}

Next we will identify a basis of $\Hcal$. Recall by Proposition \ref{Prop:Tt0Types} that $\tau$ can be realized inside of a $\Tt$-representation with some central character $\chi$. (The choice of $\chi$ is not unique but we fix one.) For $y\in \Yt$ define $f_{y,\chi}\in \Hcal_{y}$ to be the function whose support is $\Tt_0\,\ybb(\varpi)\,\Tt_0$ such that
\[
f_{y,\chi}(\tbb_1\,\ybb(\varpi)\,\tbb_2)=\tau(\tbb_1)\,\chi(\ybb(\varpi))\,\tau(\tbb_2),\quad \tbb_1, \tbb_2\in\Tt_0,
\]
where $\chi(\ybb(\varpi))$ is viewed as a scalar in $\End(\tau)$.

\begin{Lem}
    The function $f_{y,\chi}$ above is well-defined.
\end{Lem}
\begin{proof}
Note that for each $t\in Z(\Tt)$, the function $f$ on $\Tt_0\; t\;\Tt_0$ defined by
\[
f(\tbb_1\,t\,\tbb_2)=\tau(\tbb_1)\,\chi(t)\,\tau(\tbb_2),\quad \tbb_1, \tbb_2\in\Tt_0,
\]
is well-defined. Indeed, if $t= \gamma_{1} t \gamma_{2}$ for some $\gamma_{j}\in \Tt_{0}$ then we must have $\gamma_{1}\gamma_{2}=1$. Since $\chi(t)$ is central in $\mathrm{End}(\tau)$, we have $\chi(t) = \tau(\gamma_{1})\chi(t)\tau(\gamma_{2})$ as elements of $\mathrm{End}(\tau)$. Since $\ybb(\varpi)\in Z(\Tt)$, $f_{y,\chi}$ is well-defined.
\end{proof}

\begin{Lem}
    The set $\{f_{y,\chi}|y\in \Yt\}$ is a basis for $\Hcal$.
\end{Lem}
\begin{proof}
    As the functions are supported on distinct cosets, we see that these functions are linearly independent. By Lemma \ref{Lem:Hecke_Supp} and the group isomorphism $Y\cong \Tt/\Tt_{0}$, we see that they span $\Hcal$.
\end{proof}

Finally, let $\C[\Yt]$ be the group algebra over $\C$ generated by $\Yt$. Then we have

\begin{Thm}\label{Thm:Hecke_Alg}
    Let $y,z\in\Yt$. Then
    \begin{equation*}
        f_{y,\chi}*f_{z,\chi} = f_{y+z,\chi}.
    \end{equation*}
    Furthermore, the map 
    \[
    \C[\Yt]\rightarrow \Hcal,\quad e^y\mapsto f_{y,\chi}
    \]
    is an isomorphism of $\C$-algebras, where $e^{y}$ is the basis element indexed by $y$. Note that this isomorphism depends on the choice of central character $\chi$.
\end{Thm}

\begin{proof}
    By general properties of convolution, the support of $f_{y}*f_{z}$ is contained in the set $\Tt_{0}\ybb(\varpi)\Tt_{0}\zbb(\varpi)\Tt_{0} = \ybb(\varpi)\zbb(\varpi)\Tt_{0}$, where the equality follows because $\Tt_{0}$ is normal in $\Tt$.

    Since $\ybb(\varpi)\zbb(\varpi)$ is a lift of $(y+z)(\varpi)$, it suffices to show that 
    \begin{equation*}
        f_{y}*f_{z}(\ybb(\varpi)\zbb(\varpi)) = \chi(\ybb(\varpi)\zbb(\varpi)).
    \end{equation*}
    Let $t = \ybb(\varpi)\zbb(\varpi)$. Then we have
    \begin{align*}
        f_{y}*f_{z}(t) =& \int_{\Tt}f_{y}(u)f_{z}(u^{-1}t)du\\
        =& \int_{\ybb(\varpi)\Tt_{0}}f_{y}(u)f_{z}(u^{-1}t)du\\
        =& f_{y}(\ybb(\varpi))f_{z}(\ybb(\varpi)^{-1}t)\\
        =& \chi(\ybb(\varpi))\chi(\ybb(\varpi)^{-1}t)\\
        =& \chi(t).
    \end{align*}
    Now we can check directly that the specified map is an isomorphism, noting $e^{y+z}=e^y\cdot e^z$.
\end{proof}

We emphasize that this isomorphism is non-canonical; it depends on a choice of a genuine central character $\chi$ so that the associated representation $\sigma_{\chi}$ of $\Tt$ contains $\tau$. But the choices for $\chi$ are controlled by Lemma \ref{Lem:FactorZT0}.


\appendix
\section{Skew-symmetric forms over local rings}\label{Appendix}


Let $A$ be a commutative ring with unity and $M$ a projective module over $A$. Let
\[
b:M\times M\longrightarrow A
\]
be a bilinear map such that
\[
b(m, m')=-b(m', m).
\]
We call such $b$ a skew-symmetric form on $M$. We have an $A$-module homomorphism
\[
M\longrightarrow \Hom_A(M, A),\quad m\mapsto b(-, m).
\]
We say $b$ is non-singular if this homomorphism is an isomorphism. If $A$ is a field of odd characteristic, then a non-singular skew-symmetric form is called a symplectic form.

Now, assume $A$ is a local ring, so necessarily $M$ is free. Assume $\rank(M)=n$, so that
\[
M=Ae_1\oplus\cdots\oplus Ae_n
\]
by choosing a basis $e_1,\dots,e_n$. Consider the $n\times n$ matrix
\[
M_b:=(b(e_i, e_j)).
\]
It is then elementary to see 
\begin{Lem}\label{Lem:non-singular_form}
With the above notation and assumption, the skew-symmetric form $b$ is non-singular if and only if $\det(b(e_i,e_j))\in A^\times$.
\end{Lem}

Let $\mathfrak{m}\subset A$ be the maximal ideal. The skew-symmetric form on $M$ naturally gives rise to a skew-symmetric form
\[
\widebar{b}: M/ \mathfrak{m}M\times M/ \mathfrak{m}M\longrightarrow A/\mathfrak{m}A
\]
on the vector space $M/\mathfrak{m}M$ over the residue field $A/\mathfrak{m}A$. One can then see
\begin{Lem}
The skew-symmetric form $b$ on $M$ is non-singular if and only if the skew-symmetric form $\widebar{b}$ is non-singular.
\end{Lem}
\begin{proof}
    This follows from Lemma \ref{Lem:non-singular_form} together with the fact that $A^\times=A\smallsetminus \mathfrak{m}$ is the set of units.
\end{proof}

For each subset $S\subset M$, we define
\[
S^\perp=\{m\in M\st b(m, s)=0\quad\text{for all $s\in S$}\},
\]
which is always a submodule of $M$. Also let $N\subset M$ be a submodule of $M$. We say $N$ is non-singular if the restriction of $b$ on $N$ is non-singular.

\begin{Lem}
Let $(M, b)$ be non-singular, and $N\subset M$ a non-singular submodule. Then
\[
M=N\oplus N^\perp,
\]
and $N^\perp$ is also non-singular.
\end{Lem}
\begin{proof}
    This can be proved as in \cite[Proposition (3.2), p.10]{Baeza}.
\end{proof}

This lemma implies the following proposition which states that an orthogonal decomposition of $M/\mathfrak{m}M$ always lifts to one of $M$.
\begin{Prop}\label{P:reduction_of_A-module}
Let $(M, b)$ be non-singular, and let $\widebar{M}=M/\mathfrak{m}M$. Every orthogonal decomposition
\[
\widebar{M}=\widebar{N}\perp\widebar{N'},
\]
where $\widebar{N}$ is a non-singular subspace of $\widebar{M}$ lifts to an orthogonal decomposition
\[
M=N\perp N',
\]
where $N$ is non-singular with $\widebar{N}=N/\mathfrak{m}N$ and $\widebar{N'}=N'/\mathfrak{m}N'$.

\end{Prop}
\begin{proof} 
    This is \cite[Corollary (3.4), p.11]{Baeza}.
\end{proof}

Next, let us mention
\begin{Prop}\label{P:Milner}
Let $(V, b(-,-))$ be a nondegenerate skew-symmetric space over a field $k$. Then 
we have an orthogonal sum decomposition
\[
V=\la u_1\ra\oplus\cdots\oplus \la u_k\ra\oplus V_1\oplus\cdots\oplus V_m,
\]
where $\la u_i\ra$ is the 1-dimensional space having a basis $u_i\in V$, and each $V_i$ is the 2-dimensional skew-symmetric space, so it is represented by the matrix
\[
\begin{pmatrix}
    0&-1\\1&0
\end{pmatrix}.
\]
Here, note that if $\operatorname{char}(k)\neq 2$ then there is no 1-dimensional piece.
\end{Prop}
\begin{proof}
The case for $\operatorname{char}(k)\neq 2$ is well-known. For the case $\operatorname{char}(k)=2$, see \cite[p.61-62]{Milnor}, noting that in this case a skew-symmetric form is a symmetric form.
\end{proof}

Now, we can prove the main theorem of this appendix as follows.
\begin{Thm}\label{T:decomposition_of_skew-symmetric-module}
Let $M$ be a free module over a local ring $A$ with the maximal ideal $\mathfrak{m}$ equipped with a non-singular skew-symmetric form $b$. Then we have an orthogonal sum decomposition
\[
M=\la u_1\ra\oplus\cdots\oplus\la u_k\ra\;\oplus\;\la e_1, f_1\ra\oplus\cdots\oplus\la e_m, f_m\ra,
\]
where $\la u_i\ra$ is a rank 1 module spanned by some $u_i\in M$ such that $2b(u_i, u_i)=0$, and $\la e_i, f_i\ra$ is a rank 2 module spanned by $\{e_i, f_i\}\subset M$ whose associated matrix is
\[
\begin{pmatrix}
    a_i&-1-m_i\\ 1+m_i&b_i
\end{pmatrix}
\quad\text{with}\quad a_i, b_i, m_i\in \mathfrak{m},
\]
where $2a_i=2b_i=0$.

Further, the rank 1 module $\la u_i\ra$ can appear only when $\operatorname{char}(A)=2$ (namely $2=0$ in $A$).
\end{Thm}
\begin{proof}
Decompose $\widebar{M}=M/\mathfrak{m}M$ by using Proposition \ref{P:Milner} and lift the decomposition to $M$ by using Proposition \ref{P:reduction_of_A-module}. If the rank 1 module $\la u_i\ra$ appears, then we must have $b(u_i,u_i)\in A\smallsetminus \mathfrak{m}$. Since $b$ is skew-symmetric, we must have $2b(u_i, u_i)=0$, which also implies $2=0$. Similarly, we have $2a_i=2b_i=0$.
\end{proof}

\bibliographystyle{amsalpha}
\bibliography{Hecke_algebra2_bio}

\end{document}